\documentclass[12pt]{amsart}
\advance\hoffset by -0.5 truecm
\usepackage{amsmath,amscd}
\usepackage{amssymb}
\usepackage{amsthm}

\usepackage{mathrsfs}
\usepackage{hyperref}

\newtheorem{Theorem}{Theorem}[section]
\newtheorem{TheoremIntro}[Theorem]{Theorem}

\newtheorem{Lemma}[Theorem]{Lemma}
\newtheorem{Corollary}[Theorem]{Corollary}

\newtheorem{Proposition}[Theorem]{Proposition}

\newtheorem{Definition}[Theorem]{Definition}
\newtheorem*{DefinitionNoNumber}{Definition}

\newtheorem{Remark}[Theorem]{Remark}

\numberwithin{equation}{section}

\def \dim{{\mbox {dim}}\,}

\def\V{\mbox{Var}}

\def\mZ{{\mathbb Z}}
\def\R\re
\def\V{\bf V}

\def \re{{\mathbb R}}
\def \mR{{\mathbb R}}

\def \C{{\mathbb C}}

\def \V{{\bf V}}

\newcommand{\ra}{\text{rank}}

\newcommand{\abs}[1]{\lvert #1 \rvert}
\newcommand{\norm}[1]{\lVert #1 \rVert}
\newcommand{\br}[1]{\langle #1 \rangle}

\newcommand{\etilde}{\,\tilde{\rule{0pt}{6pt}}\,}

\newcommand{\eps}{\varepsilon}

\newcommand{\mDp}{\mathcal{D}'}

\newcommand{\dbar}{\overline{\partial}}

\def \vd{\overset{\tt{v}}{\nabla}}
\def \hd{\overset{\tt{h}}{\nabla}}
\def \vdiv{\overset{\tt{v}}{\mbox{\rm div}}}
\def \hdiv{\overset{\tt{h}}{\mbox{\rm div}}}

\newcounter{sidenote}
\begin{document}
\title[Invariant distributions and Beurling transforms]{Invariant distributions, Beurling transforms and tensor tomography in higher dimensions}

\author[G.P. Paternain]{Gabriel P. Paternain}
\address{ Department of Pure Mathematics and Mathematical Statistics,
University of Cambridge,
Cambridge CB3 0WB, UK}
\email {g.p.paternain@dpmms.cam.ac.uk}

\author[M. Salo]{Mikko Salo}
\address{Department of Mathematics and Statistics, University of Jyv\"askyl\"a}
\email{mikko.j.salo@jyu.fi}

\author[G. Uhlmann]{Gunther Uhlmann}
\address{Department of Mathematics, University of Washington and University of Helsinki}

\email{gunther@math.washington.edu}




\begin{abstract}
In the recent articles \cite{PSU1,PSU3}, a number of tensor tomography results were proved on two-dimensional manifolds. The purpose of this paper is to extend some of these methods to manifolds of any dimension. A central concept is the surjectivity of the adjoint of the geodesic ray transform, or equivalently the existence of certain distributions that are invariant under the geodesic flow. We prove that on any Anosov manifold, one can find invariant distributions with controlled first Fourier coefficients. The proof is based on subelliptic type estimates and a Pestov identity. We present an alternative construction valid on manifolds with nonpositive curvature, based on the fact that a natural Beurling transform on such manifolds turns out to be essentially a contraction. Finally, we obtain uniqueness results in tensor tomography both on simple and Anosov manifolds that improve earlier results by assuming a condition on the terminator value for a modified Jacobi equation.

\end{abstract}

\maketitle

\tableofcontents

\newpage

\section{Introduction} \label{sec_intro}

The unifying theme for most results in this paper is that of \emph{invariant distributions} (distributions invariant under the geodesic flow). We also discuss related harmonic analysis and dynamical concepts, and the applications of these ideas in geometric inverse problems. The paper employs results from several different areas. For the benefit of the reader, we begin with a brief overview of the topics that will appear:

\begin{itemize}
\item 
\emph{Geodesic flows.} The basic setting is a closed oriented Riemannian manifold $(M,g)$. If $SM$ is the unit tangent bundle, the geodesic flow $\phi_t$ is a dynamical system on $SM$ generated by the geodesic vector field $X$ on $SM$. We consider situations where the geodesic flow displays chaotic behavior (ergodicity etc.), with $(M,g)$ satisfying conditions such as: negative or nonpositive sectional curvature, Anosov geodesic flow, no conjugate points, or rank one conditions. \\
\item 
\emph{Invariant distributions.} The first item of interest are distributions on $SM$ that are invariant under geodesic flow, i.e.\ distributional solutions of $Xw = 0$ in $SM$. If the geodesic flow is ergodic, there are no nonconstant solutions in $L^2(SM)$, but suitable distributional solutions turn out to be useful. We give two constructions of invariant distributions with controlled first Fourier coefficients: one on nonpositively curved manifolds (via a Beurling transform), and one on manifolds with Anosov geodesic flow (via subelliptic $L^2$ estimates). \\
\item 
\emph{Harmonic analysis on $SM$.} The constructions of invariant distributions are based on harmonic analysis on $SM$. We will consider spherical harmonics (or Fourier) expansions of functions in $L^2(SM)$ with respect to vertical variables. The geodesic vector field $X$ has a splitting $X = X_+ + X_-$ which respects these expansions. The interplay of $\{X, \hd, \vd\}$, where $\hd$ and $\vd$ are horizontal and vertical gradients in $SM$, leads to the basic $L^2$ identity (Pestov identity). \\
\item 
\emph{Dynamical aspects.} For $L^2$ estimates, we need certain notions that interpolate between the ''no conjugate points'' and ''nonpositive curvature'' conditions: $\alpha$-controlled manifolds, $\beta$-conjugate points, and terminator value $\beta_{Ter}$. These conditions are studied from the dynamical systems point of view, leading to a new characterization of manifolds with Anosov geodesic flow. \\
\item 
\emph{Geometric inverse problems.}
Our main motivation for studying the above questions comes from several geometric inverse problems related to the ray transform $I_m$, which encodes the integrals of a symmetric $m$-tensor field over all closed geodesics. Basic questions include the injectivity of $I_m$ (up to natural obstruction) and, dually, the surjectivity of $I_m^*$. Our results on invariant distributions lead to surjectivity of $I_m^*$, and consequently to injectivity of $I_m$ and spectral rigidity results, under various conditions. Analogous results are discussed also on compact manifolds with boundary.
\end{itemize}

This paper merges and expands ideas developed in \cite{GK,GK1,GK2,Sh,DS,PSU1,PSU3}. We will now begin a more detailed discussion.

\vspace{15pt}

\noindent {\bf Invariant distributions.}\ Let $(M,g)$ be a compact oriented Riemannian manifold with or without boundary, and let $X$ be the geodesic vector field regarded as a first order differential operator $X:C^{\infty}(SM)\to C^{\infty}(SM)$ acting on functions on the unit tangent bundle $SM$. This paper is concerned with weak solutions $u$ to the transport equation $Xu=f$ under various conditions on the metric $g$. A good understanding of the transport equation is essential in many geometric inverse problems such as tensor tomography, boundary rigidity and spectral rigidity. It also plays a fundamental role in the study of the dynamics of the geodesic flow. For example, ergodicity of the geodesic flow on a closed manifold is just the absence of non-constant $L^2$ functions $w$ with $Xw=0$. Similarly, certain solvability results for the equation $Xu = f$ on a manifold with boundary are equivalent with nontrapping conditions for the geodesic flow \cite{DuistermaatHormanderII}.

For the most part, we will consider closed manifolds $(M,g)$. In this paper we are going to exhibit many \emph{invariant distributions} (distributional solutions to $Xw=0$) having prescribed initial Fourier components and Sobolev regularity $H^{-1}$. This will be possible when $(M,g)$ is either Anosov or a rank one manifold of non-positive sectional curvature. Both cases contain in particular manifolds of negative sectional curvature. Invariant distributions 
will be constructed by two mechanisms: one emanates from a fundamental $L^2$-identity, called the Pestov identity, and subelliptic type estimates for certain operators associated with the transport equation. The other mechanism consists in the introduction of a Beurling transform acting on trace-free symmetric tensor fields. The key property of this transform, also proved via the Pestov identity, is that it is essentially a contraction when the sectional curvature is non-positive.

Let us recall that the rank, $\ra(v)$, of a unit vector $v\in T_{x}M$ is the dimension of the space of parallel Jacobi fields along the geodesic $\gamma$ determined by $(x,v)\in SM$. (A parallel Jacobi field is a parallel vector field that also satisfies the Jacobi equation.) Since $\dot{\gamma}$ is trivially a parallel Jacobi field, $\ra(v)\geq 1$ and $\ra(v)\geq 2$ if and only if there is a non-zero parallel Jacobi field $J$ along $\gamma$ which is orthogonal to $\dot{\gamma}$. Given such a field, the sectional curvature of the 2-plane spanned by $J(t)$ and $\dot{\gamma}(t)$ is zero for all $t$.
The rank of $(M,g)$ is defined as the minimum of $\ra(v)$ over all $v$.

Our metrics will be free of conjugate points, but we shall require at various points enhanced conditions in this direction. To quantify these stronger conditions we introduce following \cite{Pestov,Dairbekov_nonconvex,PSU3} a modified Jacobi equation. Let $R$ denote the Riemann curvature tensor and let $\beta\in\re$.

\begin{DefinitionNoNumber} Let $(M,g)$ be a Riemannian manifold. We say that $(M,g)$ is free of $\beta$-conjugate points if for any geodesic $\gamma(t)$, all nontrivial solutions of the equation $\ddot{J} + \beta R(J,\dot{\gamma}) \dot{\gamma} = 0$ with $J(0) = 0$ only vanish at $t = 0$. The \emph{terminator value} of $(M,g)$ is defined to be 
$$
\beta_{Ter} = \sup \,\{ \beta \in [0,\infty]:\, \text{$(M,g)$ is free of $\beta$-conjugate points} \}.
$$
\end{DefinitionNoNumber}

Clearly $1$-conjugate points correspond to conjugate points in the usual sense. For a closed manifold $(M,g)$, we will show in Section \ref{sec_betaconjugate} that 
\begin{itemize}
\item if $(M,g)$ is free of $\beta_0$-conjugate points for some $\beta_0 > 0$, then $(M,g)$ is free of $\beta$-conjugate points for $\beta \in [0,\beta_0]$, 
\item 
$(M,g)$ is Anosov if and only if $\beta_{Ter} > 1$ and every unit vector has rank one (see Corollary \ref{corollary:ganosov} below; this seems to be a new geometric characterization of the Anosov property generalizing \cite[Corollary 3.3]{Ebe});
\item 
$(M,g)$ has nonpositive curvature if and only if $\beta_{Ter} = \infty$, see Lemma \ref{lemma:kbeta} below.
\end{itemize}

Recall that a closed manifold $(M,g)$ is said to be Anosov, if its geodesic flow $\phi_t$ is Anosov. This means that there
is a continuous invariant splitting
$TSM=E^0\oplus E^{u}\oplus E^{s}$, where $E^0$ is the flow direction, and 
there are constants $C>0$ and $0<\rho<1<\eta$ such that 
for all $t>0$ 
\[\|d\phi_{-t}|_{E^{u}}\|\leq C\,\eta^{-t}\;\;\;\;\mbox{\rm
and}\;\;\;\|d\phi_{t}|_{E^{s}}\|\leq C\,\rho^{t}.\]
Another relevant characterization of this property states that  $(M,g)$ is Anosov if and only if
the metric $g$ lies in the $C^2$-interior of the set of metrics without conjugate points \cite{Ru}.

By considering the vertical Laplacian $\Delta$ on each fibre $S_{x}M$ of $SM$ we have a natural $L^2$-decomposition $L^{2}(SM)=\oplus_{m\geq 0}H_m(SM)$
into vertical spherical harmonics. We also set $\Omega_m=H_{m}(SM)\cap C^{\infty}(SM)$. Then a function $u$ belongs to $\Omega_m$ if and only if
$\Delta u=m(m+n-2)u$ where $n=\dim M$. There is a natural identification between functions in $\Omega_m$ and trace-free symmetric $m$-tensors (for details on this see \cite{GK2,DS} and Section \ref{sec:vlap+sh}). The geodesic vector field $X$ maps $\Omega_m$ to
$\Omega_{m-1}\oplus\Omega_{m+1}$ and hence we can split it as $X=X_{+}+X_{-}$, where $X_{\pm}:\Omega_{m}\to \Omega_{m\pm 1}$
and $X_{+}^*=-X_{-}$. The operator $X_{+}$ is overdetermined elliptic, i.e.~it has injective symbol, and its finite dimensional kernel consists, under the identification mentioned above, of {\it conformal Killing tensor fields} of rank $m$. The operator $X_{-}$ is of divergence type.

Assume that $(M,g)$ has no nontrivial conformal Killing tensors (this holds on any surface of genus $\geq 2$
and more generally on any manifold whose conformal class contains a negatively curved metric \cite{DS} or a rank one metric of non-positive curvature, see Corollary \ref{cor:noCKT}). Then the operator $X_{-}:\Omega_{m}\to\Omega_{m-1}$ is surjective for $m\geq 2$, so one could attempt to find solutions to the transport equation $Xw=0$ as follows. Start with an element $w_k \in\Omega_k$ with $X_{-}w_k=0$ (i.e.~a solenoidal trace-free symmetric $k$-tensor).
Surjectivity of $X_{-}$ means that there is a unique $w_{k+2}\in\Omega_{k+2}$ such that $w_{k+2}$ is orthogonal to $\text{Ker}\, X_{-}$ (equivalently, $w_{k+2}$ has minimal $L^2$ norm) and
\[X_{-}w_{k+2}+X_{+}w_k=0.\]
Continuing in this fashion we construct a formal solution $w=w_k+w_{k+2}+\dots$ to the transport equation $Xw=0$. This formal solution can be conveniently expressed using what we call the {\it Beurling transform}.

\begin{DefinitionNoNumber}
Let $(M,g)$ be a closed manifold without nontrivial conformal Killing tensors. Given $k\geq 0$ and $f_k\in\Omega_k$, there is a unique function $f_{k+2}\in\Omega_{k+2}$ orthogonal to $\mbox{\rm Ker}\,X_{-}$ such that $X_{-}f_{k+2}=-X_{+}f_{k}$.  We define the Beurling transform as
$$
B: \Omega_k \to \Omega_{k+2}, \ \ f_k \mapsto f_{k+2}.
$$
\end{DefinitionNoNumber}

The terminology comes from the fact that if $M$ is two-dimensional, then bundles of symmetric trace-free tensors can be expressed in terms of holomorphic line bundles and $X_{\pm}$ correspond to $\partial$ and $\dbar$ operators. If $X_{\pm}$ are split in terms of the $\eta_{\pm}$ operators of Guillemin-Kazhdan \cite{GK}, then $-B$ corresponds exactly to $\dbar^{-1} \partial$ and $\partial^{-1} \dbar$ operators as in the classical Beurling transform in the complex plane \cite{AIM}. See Appendix \ref{sec_2d_case} for more details. Related transforms have been studied on differential forms in $\mR^n$ \cite{IwaniecMartin}, on Riemannian manifolds \cite{LiXD}, and in hyperbolic space \cite{Hedenmalm}. The Beurling transform considered in this paper (which is essentially the first ladder operator in \cite{GK1} when $n=2$) seems to be different from these.

Let us return to invariant distributions:

\begin{DefinitionNoNumber}
Let $(M,g)$ be a closed manifold without conformal Killing tensors, and let $f \in \Omega_{k_0}$ satisfy $X_- f = 0$. The \emph{formal invariant distribution} starting at $f$ is the formal sum 
$$
w = \sum_{j=0}^{\infty} B^j f.
$$
\end{DefinitionNoNumber}

At this point we do not know if the sum converges in any reasonable sense. However, if the manifold has nonpositive sectional curvature it converges nicely. This follows from the fact that the Beurling transform is essentially a contraction on such manifolds and this is the content of our first theorem.


\begin{TheoremIntro} \label{theoremA}
Let $(M,g)$ be a closed $n$-manifold without conformal Killing tensors and with non-positive sectional curvature. 
Then for any $m \geq 0$  we have 
$$
\norm{Bf}_{L^2} \leq C_n(m) \norm{f}_{L^2}, \quad f \in \Omega_m,
$$
where 
\begin{align*}
C_2(m) &= \left\{ \begin{array}{cl} \sqrt{2}, & m = 0, \\ 1, & m \geq 1, \end{array} \right. \\
C_3(m) &= \left[ 1 + \frac{1}{(m+2)^2(2m+1)} \right]^{1/2}, \\
C_n(m) &\leq 1 \ \text{ for $n \geq 4$.}
\end{align*}
If $m_0 \geq 0$ and if $f \in \Omega_{m_0}$ satisfies $X_- f = 0$, then the formal invariant distribution $w$ starting at $f$ is an element of $L^2_x H^{-1/2-\eps}_{v\phantom{\theta}}$ for any $\eps > 0$. Moreover, the Fourier coefficients of $w$ satisfy 
$$
\norm{w_{m_0+2k}}_{L^2} \leq A_n(m_0) \norm{f}_{L^2}, \quad k \geq 0,
$$
where $A_n(m_0) = \prod_{j=0}^{\infty} C_n(m_0+2j)$ is a finite constant satisfying 
\begin{align*}
A_2(m_0) &= \left\{ \begin{array}{cl} \sqrt{2}, & m_0 = 0, \\ 1, & m_0 \geq 1, \end{array} \right. \\
A_3(m_0) &\leq 1.13, \\
A_n(m_0) &\leq 1 \ \text{ for $n \geq 4$.}
\end{align*}

\end{TheoremIntro}


In the theorem we use the mixed norm spaces
$$
L^2_x H^{s}_{v}(SM) = \{ u \in \mDp(SM) :\,\, \norm{u}_{L^2_x H^s_{v}} < \infty \}, \ \ \norm{u}_{L^2_x H^s_{v}} = \left( \sum_{m=0}^{\infty} \langle m \rangle^{2s }\norm{u_m}_{L^2}^2 \right)^{1/2},
$$
where as usual $\langle m \rangle=(1+m^2)^{1/2}$. Note that the norm of the Beurling transform is always $\leq 1$ in dimensions $n \geq 4$, is $\leq 1$ in two dimensions unless $m=0$, and is sufficiently close to $1$ in three dimensions so that formal invariant distributions exist in nonpositive curvature. We remark that the same construction gives a more general family of distributions, where finitely many Fourier coefficients are obtained by taking any solutions of $X_{-}f_{k+2}=-X_{+}f_{k}$ (not necessarily orthogonal to $\mbox{\rm Ker}\,X_{-}$) and the remaining coefficients are obtained from the Beurling transform.

We shall see below that a rank one manifold of non-positive sectional curvature must be free of conformal Killing tensors and thus by Theorem \ref{theoremA} given any $f\in\Omega_{k_0}$ with $X_{-}f=0$ there is $w\in L^2_x H^{-1/2-\eps}_{v\phantom{\theta}}$ with $Xw=0$ and $w_{k_0}=f$. We state this explicitly in the following corollary.

\begin{Corollary} Let $(M,g)$ be a closed rank one manifold of non-positive sectional curvature.  Given $f\in\Omega_{k_0}$ with $X_{-}f=0$ there is $w\in L^2_x H^{-1/2-\eps}_{v\phantom{\theta}}$ with $Xw=0$ and $w_{k_0}=f$.
\end{Corollary}

Our next task is to discuss distributional solutions to $Xw=0$ replacing the hypothesis of sectional curvature $K\leq 0$ by the Anosov condition. This makes the problem harder but nevertheless we can show by using the Pestov identity and a duality argument:


\begin{TheoremIntro} \label{theoremB}
Let $(M,g)$ be an Anosov manifold. Given any $f\in\Omega_k$ for $k=0,1$ with $X_{-}f=0$ (if $k=0$ this is vacuously true), there exists $w\in H^{-1}(SM)$ with $Xw=0$ and $w_k=f$.
\end{TheoremIntro}

Let us compare the invariant distributions in Theorems \ref{theoremA} and \ref{theoremB} for the case $k=0$, solving the equation $Xw = 0$ with $w_0 = f_0$ for given $f_0 \in \Omega_0$. The formal invariant distributions in Theorem \ref{theoremA} exist on manifolds with $K \leq 0$, they lie in $L^2_x H^{-1/2-\eps}_{v\phantom{\theta}}$, and they are the unique invariant distributions whose Fourier coefficients $w_k$ are minimal energy solutions of $X_- w_{k+2} = -X_+ w_k$. On the other hand, by the results in Section \ref{sec_surjectivity_izero} the invariant distributions in Theorem \ref{theoremB} exist on any Anosov manifold, they lie in $L^2_x H^{-1}_v$, and they are the unique invariant distributions for which the quantity $\sum_{m=1}^{\infty} \frac{1}{m(m+n-2)} \norm{w_m}_{L^2}^2$ is minimal. It is an interesting question whether there is any relation between these two classes of invariant distributions.

Very recently, Theorem \ref{theoremB} has been improved in \cite{Gui1}. The paper \cite{Gui1} introduces a new set of tools coming from microlocal analysis of Anosov flows and it also gives information on the wave front set of the invariant distributions.

We remark that in general, an arbitrary transitive Anosov flow has a plethora of invariant measures and distributions (e.g.\ equilibrium states, cf.\ \cite{KH}), but the ones in Theorems \ref{theoremA}--\ref{theoremB} are {\it geometric} since they really depend on the geometry of the spherical fibration $\pi:SM\to M$.
In the case of surfaces of constant negative curvature these distributions and their regularity are discussed in \cite[Section 2]{AZ}.

\vspace{15pt}

\noindent {\bf Interlude.}\ One of our main motivations for considering invariant distributions as above has been the fundamental interplay that has been emerging in recent years between injectivity properties of the geodesic ray transform and the existence of special solutions to the transport equation $Xu=0$.
This interplay has lead to the solution of several long-standing geometric inverse problems, especially in two dimensions \cite{PU,SaU,PSU2,PSU1,PSU3,Gui1,Gui2}.

We will now try to explain this interplay in the easier setting of simply connected compact manifolds with strictly convex boundary and without conjugate points. Such manifolds are called {\it simple} and an obvious example is the region enclosed by a simple, closed and strictly convex curve in the plane (less obvious examples are small perturbations of this flat region).
It is well known that under these assumptions $M$ is topologically a ball and geodesics hit the boundary in finite time, i.e.\ $(M,g)$ is nontrapping.
The notion of simple manifold appears naturally in the context of the boundary rigidity problem \cite{Mi} and it has been at the center of recent activity on geometric inverse problems.

Geodesics going from $\partial M$ into $M$ are parametrized by $\partial_+ (SM) = \{(x,v) \in SM \,;\, x \in \partial M, \langle v,\nu \rangle \leq 0 \}$ where $\nu$ is the outer unit normal vector to $\partial M$. For $(x,v) \in SM$ we let $t \mapsto \gamma(t,x,v)$ be the geodesic starting from $x$ in direction $v$. The ray transform of $f\in C^{\infty}(SM)$ is defined by 
$$
If(x,v):= \int_0^{\tau(x,v)} f(\phi_t(x,v)) \,dt, \quad (x,v) \in \partial_+(SM),
$$
where $\tau(x,v)$ is the exit time of $\gamma(t,x,v)$. Even though the
definition is given for smooth functions, $I$ acting on $L^2$-functions is well-defined too. Integrating or averaging along the orbits of a group action is the obvious way to produce invariant objects and this case is no exception: $If$ naturally gives rise to a first integral of the geodesic flow since it is a function defined on $\partial_+ (SM)$ which parametrizes the orbits of the geodesic flow. In general, if $h\in C(\partial_+ (SM))$, then $h^{\sharp}(x,v):=h(\phi_{-\tau(x,-v})(x,v))$ is a first integral. There are natural $L^{2}$-inner products so that we can consider the adjoint $I^{*}:L^{2}(\partial_+ (SM))\to L^{2}(SM)$ which turns out to be $I^*h=h^{\sharp}$.

The situation gets interesting once we start restricting $I$ to relevant subspaces of $L^{2}(SM)$, a typical example being $L^{2}(M)$, and in this case the resulting geodesic ray transform is denoted by $I_{0}$ and $I_{0}^*h$ is easily seen to be the Fourier coefficient of zero degree of $h^{\sharp}$. Hence surjectivity of $I_{0}^*$ can be interpreted
as the existence of a solution to $Xu=0$ with prescribed $u_0$. In the context of simple manifolds a fundamental property is that $I^*_{0}I_{0}$ is an elliptic pseudo-differential operator of order $-1$ in the interior of $M$ \cite{PU}
and this combined with injectivity results for $I_{0}$ gives rise to the desired surjectivity properties for $I_{0}^*$.

Symmetric tensors can also be seen as interesting subspaces of $L^{2}(SM)$ and the same ideas apply.  Sometimes it is convenient to use this interplay backwards: if one can construct invariant distributions with certain prescribed Fourier components one might be able to prove injectivity of the relevant geodesic ray transform. In the case of Anosov manifolds the situation is technically more challenging, but the guiding principles remain and this was exploited and explained at great length
in \cite{PSU3}. In fact
Theorem \ref{theoremB} extends \cite[Theorem 1.4 and Theorem 1.5]{PSU3} to any dimension, and the result should be regarded as the analogue of the surjectivity result for the adjoint of the geodesic ray transform on simple manifolds acting on functions
and 1-forms \cite{PU,DU}.  

It seems convenient at this point to conclude the interlude and give details about how to define the
geodesic ray transform in the context of closed manifolds.

\vspace{15pt}

\noindent {\bf Inverse problems on closed manifolds.}\ Let $(M,g)$ be a closed oriented manifold, and let $\mathcal{G}$ be the set of periodic geodesics parametrized by arc length. The ray transform of a symmetric $m$-tensor field $f$ on $M$ is defined by 
$$
I_{m}f(\gamma) = \int_0^T f(\dot{\gamma}(t), \ldots, \dot{\gamma}(t)) \,dt, \quad \gamma \in \mathcal{G} \text{ has period } T.
$$
It is easy to check that $I_{m}(dh)(\gamma) = 0$ for all $\gamma \in \mathcal{G}$ if $h$ is a symmetric $(m-1)$-tensor and $d$ denotes the symmetrized Levi-Civita covariant derivative acting on symmetric tensors.
The tensor tomography problem asks whether these are the only tensors in the kernel of $I_m$. When this occurs $I_m$ is said to be \emph{solenoidal injective} or \emph{$s$-injective}. Of course, one would expect a positive answer only on manifolds $(M,g)$ with sufficiently many periodic geodesics. The Anosov manifolds are one reasonable class where this question has been studied.

Our main result in this direction is:

\begin{TheoremIntro} \label{theoremC}
Let $(M,g)$ be a closed Riemannian manifold such that every unit vector has rank one.
 Suppose in addition that $\beta_{Ter}>\frac{m(m+n-1)}{2m+n-2}$ where $m$ is an integer $\geq 2$ and $n=\dim M$.
Then $I_{m}$ is $s$-injective. 
\end{TheoremIntro}


Note that the hypotheses imply by Corollary \ref{corollary:ganosov} that $(M,g)$ is an Anosov manifold.
In \cite{DS0} it was shown that $I_0$ and $I_{1}$ are $s$-injective on any Anosov manifold.
Theorem \ref{theoremC} was proved earlier for negative curvature in \cite{GK}  for $n=2$ and in \cite{CS} for arbitrary $n$. For an arbitrary Anosov surface, $s$-injectivity of $I_2$ was established in \cite{PSU3},
and earlier the same result was proved in \cite{SU}  if additionally the surface has no focal points. Solenoidal injectivity of $I_m$ for an Anosov surface and $m\geq 3$ is finally settled in \cite{Gui1}.
It is also known that for any $n$ and $m$ the kernel of $I_m$ is finite dimensional on any Anosov manifold \cite{DS0}. 
Note that for $m=2$ the condition in Theorem \ref{theoremC} becomes $\beta_{Ter}>\frac{2(n+1)}{2+n}$ and hence
for any Anosov manifold with $\beta_{Ter}\in [2,\infty]$, $I_{2}$ is $s$-injective.


There are numerous motivations for considering the tensor tomography problem for an Anosov manifold but perhaps the most notorious one is that of spectral rigidity which involves $I_{2}$. In \cite{GK} Guillemin and Kazhdan proved that
if $(M,g)$ is an Anosov manifold such that $I_{2}$ is $s$-injective then $(M,g)$ is spectrally rigid. This means that if $(g_s)$ is a smooth family of Riemannian metrics on $M$ for $s \in (-\varepsilon,\varepsilon)$ such that $g_0 = g$ and the spectra of $-\Delta_{g_s}$ coincide up to multiplicity,
$$
\text{Spec}(-\Delta_{g_s}) = \text{Spec}(-\Delta_{g_0}), \quad s \in (-\varepsilon,\varepsilon),
$$
then there exists a family of diffeomorphisms $\psi_s: M \to M$ with $\psi_0 = \text{Id}$ and 
$$
g_s = \psi_s^* g_0.
$$
Hence directly from Theorem \ref{theoremC} we deduce:


\begin{TheoremIntro} \label{theoremCor}
Let $(M,g)$ be a closed Riemannian manifold such that every unit vector has rank one.
 Suppose in addition that $\beta_{Ter}>\frac{2(n+1)}{n+2}$ where $n=\dim M$.
Then $(M,g)$ is spectrally rigid.
\end{TheoremIntro}


\vspace{5pt}

\noindent {\bf Inverse problems on manifolds with boundary.}\ The results up to this point have been about closed manifolds, but our techniques also apply to the case of simple manifolds with boundary.  In fact in \cite{PSU3} we advocated a strong analogy between simple and Anosov manifolds and this has proved quite fruitful as we hinted in the interlude.
The tensor tomography problem on simple manifolds is well known and we refer to \cite{Sh,PSU4} for extensive discussions.
The geodesic ray transform of a symmetric tensor $f$ is defined by 
$$
I_mf(x,v) = \int_0^{\tau(x,v)} f(\phi_t(x,v)) \,dt, \quad (x,v) \in \partial_+(SM),
$$
where we abuse notation and we denote by $f$ also the function $(x,v)\mapsto f_{x}(v,\dots,v)$.
If $h$ is a symmetric $(m-1)$-tensor field with $h|_{\partial M} = 0$, then 
$I_m(dh) = 0$. The transform $I_m$ is said to be \emph{$s$-injective} if these are the only elements in the kernel. The terminology arises from the fact that
any tensor field $f$ may be written uniquely as $f=f^s+dh$, where $f^s$
is a symmetric $m$-tensor with zero divergence and $h$ is an $(m-1)$-tensor
with $h|_{\partial M} = 0$ (cf.~\cite{Sh}). The tensor fields $f^s$ and
$dh$ are called respectively the {\it solenoidal} and {\it potential} parts
of $f$. Saying that $I_m$ is $s$-injective is saying precisely that
$I_m$ is injective on the set of solenoidal tensors.

We can now state our last theorem.


\begin{TheoremIntro} \label{theoremD}
Let $(M,g)$ be a compact simple manifold with
$\beta_{Ter}\geq \frac{m(m+n-1)}{2m+n-2}$ where $n = \dim M$ and $m$ is an integer $\geq 2$.
Then $I_m$ is $s$-injective.
\end{TheoremIntro}


This theorem 
improves the results in \cite{Pestov, Dairbekov_nonconvex} in which $s$-injectivity of $I_m$ is proved under the weaker condition $\beta_{Ter}\geq \frac{m(m+n)}{2m+n-1}$. It is also known that $I_m$ is always $s$-injective on simple surfaces \cite{PSU1}, and $I_2$ is $s$-injective for a generic class of simple metrics including real-analytic ones \cite{SU3}, but it remains an open question whether $I_2$ is $s$-injective on arbitrary simple manifolds of dimension $\geq 3$ (see however the recent papers \cite{SUV2,Gui2} for results under different conditions).

\vspace{10pt}

\noindent {\bf Open questions, structure of the paper.}\  Here we list some open questions related to the topics of this article:
\begin{itemize}
\item 
Is $I_2$ $s$-injective on simple manifolds with $\dim M \geq 3$?
\item 
Is $I_2$ $s$-injective on Anosov manifolds with $\dim M \geq 3$?
\item 
Do Anosov manifolds with $\dim M \geq 3$ have no nontrivial conformal Killing tensors? Can one find more general conditions to ensure this?
\item 
Is there any relation between the formal and minimal invariant distributions in Theorems \ref{theoremA} and \ref{theoremB}?
\item 
Can one say more about the norm of the Beurling transform?
\end{itemize}

Much of this paper deals with multidimensional generalizations of the results in \cite{PSU1,PSU3}. The paper is organized as follows. Section \ref{sec_intro} is the introduction and states the main results. In Section \ref{sec_spherebundle} we prove the Pestov identity following the approach of \cite{PSU1}. The result is well known, but we give a simple proof based on three first order operators on $SM$ and their commutator formulas which are multidimensional analogues of the structure equations on the circle bundle of a Riemannian surface. Section \ref{sec:vlap+sh} follows \cite{GK2, DS} and discusses spherical harmonic expansions in the vertical variable, which generalize Fourier expansions in the angular variable for $\dim M = 2$. Section \ref{sec_alphacontrolled} considers the notion of $\alpha$-controlled manifolds (introduced in \cite{PSU3} for surfaces), and contains certain useful estimates.

In Section \ref{section_beurling} we consider the Beurling transform and show that it is essentially a contraction on nonpositively curved manifolds. Section \ref{sec_betaconjugate} studies $\beta$-conjugate points and terminator values, and Section \ref{section_anosov_alphacontrolled} shows that $\beta_{Ter} \geq \beta$ implies $(\beta-1)/\beta$-controlled. Sections \ref{sec_surjectivity_izero} and \ref{sec_surjectivity_im} are concerned with subelliptic estimates coming from the Pestov identity and the existence of invariant distributions related to surjectivity of $I_m^*$, and in Section \ref{section_injectivity} we discuss solenoidal injectivity for $I_m$. All the work up to this point has been on closed manifolds; Section \ref{sec_boundary} considers analogous results on compact manifolds with boundary. There are two appendices. The first appendix contains local coordinate formulas for operators arising in this paper and proves the basic commutator formulas. The second appendix considers the results of this paper specialized to the two-dimensional case, mentions a relation between Pestov and Guillemin-Kazhdan energy identities, and connects the present treatment to the works \cite{PSU1, PSU3}.

\medskip

\noindent{\it Acknowledgements.}
M.S. was supported in part by the Academy of Finland and an ERC starting grant, and G.U. was partly supported by NSF and a Simons Fellowship. The authors would like to express their gratitude to the Banff International Research Station (BIRS) for providing an excellent research environment via the Research in Pairs program and the workshop Geometry and Inverse Problems, where part of this work was carried out. We are also grateful to Hanming Zhou for several corrections to earlier drafts, and to Joonas Ilmavirta for helping with a numerical calculation.
Finally we thank the referee for numerous suggestions that improved the presentation.

\section{Commutator formulas and Pestov identity} \label{sec_spherebundle}

In this section we introduce and prove the fundamental energy identity which is the basis of a considerable part of our work.
This identity is a special case of a more general identity already in the literature. However we take a slightly different approach to its derivation which emphasizes the role of the unit sphere bundle very much in the spirit of \cite{PSU1}.  For other presentations see \cite{Sh}, \cite{Kn}, \cite[Theorem 4.8]{DP1}.

Let $(M,g)$ be a closed Riemannian manifold with unit sphere bundle $\pi: SM\to M$ and as always let $X$ be the geodesic vector field.
It is well known that $SM$ carries a canonical metric called the Sasaki metric. If we let $\mathcal V$ denote the vertical subbundle
given by $\mathcal V=\mbox{\rm Ker}\,d\pi$, then there is an orthogonal splitting with respect to the Sasaki metric:
\[TSM=\re X\oplus {\mathcal H}\oplus {\mathcal V}.\]
The subbundle ${\mathcal H}$ is called the horizontal subbundle. Elements in $\mathcal H(x,v)$ and $\mathcal V(x,v)$ are canonically
identified with elements in the codimension one subspace $\{v\}^{\perp}\subset T_{x}M$. We shall use this identification freely below (see for example \cite{Kn,Pa} for details on these facts).

Given a smooth function $u\in C^{\infty}(SM)$ we can consider its gradient $\nabla u$ with respect to the Sasaki metric. Using the splitting
above we may write uniquely
\[\nabla u=((Xu)X,\hd u,  \vd u). \]
The derivatives $\hd u$ and $\vd u$ are called horizontal and vertical derivatives respectively. Note that this differs slightly
from the definitions in \cite{Kn,Sh} since here we are considering all objects defined on $SM$ as opposed to $TM$.
One advantage of this is to make more transparent the connection with our approach in \cite{PSU1} for the two-dimensional case.

We shall denote by $\mathcal Z$ the set of smooth functions $Z:SM\to TM$ such that $Z(x,v)\in T_{x}M$ and $\langle Z(x,v),v\rangle=0$ for all
$(x,v)\in SM$. With the identification mentioned above we see that $\hd u,\vd u\in \mathcal Z$.

Observe that $X$ acts on $\mathcal Z$ as follows:
\[XZ(x,v):=\frac{DZ(\phi_{t}(x,v))}{dt}|_{t=0}\]
where $\phi_t$ is the geodesic flow. Note that $Z(t):=Z(\phi_{t}(x,v))$ is a vector field along the geodesic $\gamma$ determined by $(x,v)$, so it makes sense to take its covariant derivative with respect to the Levi-Civita connection of $M$. Since $\langle Z,\dot{\gamma}\rangle=0$ it follows that
$\langle \frac{DZ}{dt},\dot{\gamma}\rangle=0$ and hence $XZ\in \mathcal Z$.

Another way to describe the elements of $\mathcal Z$ is a follows. Consider the pull-back bundle
$\pi^*TM\to SM$.  Let $N$ denote the subbundle of $\pi^*TM$ whose fiber over $(x,v)$
is given by $N_{(x,v)}=\{v\}^{\perp}$. Then $\mathcal Z$ coincides with the smooth sections
of the bundle $N$. Observe that $N$ carries a natural $L^{2}$-inner product and with respect to this product the formal adjoints of $\vd:C^{\infty}(SM)\to\mathcal Z$ and $\hd:C^{\infty}(SM) \to \mathcal Z$ are denoted by $-\vdiv$ and $-\hdiv$ respectively. Note that since $X$ leaves invariant the volume form of the Sasaki metric we have $X^*=-X$ for both actions of $X$ on $C^{\infty}(SM)$ and $\mathcal Z$.

The next lemma contains the basic commutator formulas. The first three of these formulas are the analogues of the structure equations used in \cite{PSU1} in two dimensions. Let $R(x,v):\{v\}^{\perp}\to \{v\}^{\perp}$ be the operator determined by the Riemann curvature tensor $R$ by $R(x,v)w=R_{x}(w,v)v$ and let $n=\dim M$.

\begin{Lemma} \label{lemma_basic_commutator_formulas}
The following commutator formulas hold on $C^{\infty}(SM)$:
\begin{align*}
[X,\vd]&=-\hd, \\ 
[X,\hd]&=R\,\vd, \\ 
\hdiv\,\vd-\vdiv\,\hd&=(n-1)X. 
\end{align*}
Taking adjoints, we also have the following commutator formulas on $\mathcal Z$:
\begin{align*}
[X,\vdiv] &= -\hdiv, \\
[X,\hdiv] &= -\vdiv R.
\end{align*}
\end{Lemma}

These commutator formulas can be extracted from the calculus of semibasic tensor fields \cite{Sh, DS}. For completeness we also prove them in Appendix \ref{sec_localcoordinates}, which contains local coordinate expressions for many operators arising in this paper.

We next prove the Pestov identity, following the approach of \cite{PSU1}. We briefly recall the motivation for this identity which comes from showing that the ray transform $I_0$ on Anosov manifolds is injective. If $f \in C^{\infty}(M)$ satisfies $I_0 f = 0$, meaning that the integrals of $f$ over periodic geodesics vanish, then the Livsic theorem \cite{dLMM} implies that there exists $u \in C^{\infty}(SM)$ with $Xu = f$. Since $f$ only depends on $x$, we have $\vd f = 0$, and consequently $u$ is a solution of $\vd X u = 0$ in $SM$.

The Pestov identity is the following energy estimate involving the norm $\norm{\vd X u}^2$. It implies that a smooth solution of $\vd X u = 0$ on an Anosov manifold must be constant, and thus by the above argument any smooth function $f$ on $M$ with $I_0 f = 0$ must be zero. All norms and inner products will be $L^2$.

\begin{Proposition} \label{prop_pestov}
Let $(M,g)$ be a closed Riemannian manifold. Then
\begin{equation*}
\norm{\vd X u}^2 = \norm{X\vd u}^2-(R\,\vd u,\vd u) + (n-1)\norm{Xu}^2
\end{equation*}
for any $u\in C^{\infty}(SM)$.
\end{Proposition}
\begin{proof}
For $u \in C^{\infty}(SM)$, the commutator formulas in Lemma \ref{lemma_basic_commutator_formulas} imply that 
\begin{align*}
\norm{\vd X u}^2 - \norm{X\vd u}^2 &= (\vd Xu, \vd Xu) - (X \vd u, X \vd u) \\
 &= ( (X \vdiv \vd X - \vdiv X X \vd)u, u) \\
 &= ( (-\hdiv \vd X + \vdiv X \hd)u, u) \\
 &=( (- \hdiv \vd X + \vdiv \hd X + \vdiv R \vd) u, u) \\
 &= -(n-1)( X^2 u, u) + (\vdiv R \vd u, u) \\
 &= (n-1) \norm{Xu}^2 - (R \vd u, \vd u).
\end{align*}
This is the required estimate.
\end{proof}

\begin{Remark}{\rm The same identity holds, with the same proof, for $(M,g)$ a compact Riemannian manifold with boundary provided that $u|_{\partial(SM)}=0$.}
\end{Remark}

\section{Spherical harmonics expansions}
\label{sec:vlap+sh}

\noindent {\bf Vertical Laplacian.} In this section we consider the vertical Laplacian 
$$
\Delta:C^{\infty}(SM)\to C^{\infty}(SM), \ \ \Delta:=-\vdiv\vd.
$$
If we fix a point $x\in M$ and consider $S_{x}M$ with the inner product determined by $g_{x}$, then $(\Delta u)(x,v)$ coincides with the Laplacian of the function $v\mapsto u(x,v)$ on the manifold $(S_{x}M,g_{x})$. 

Let us recall that for $S^{n-1}$ endowed with the canonical metric, the Laplacian $\Delta_{S^{n-1}}$
has eigenvalues $\lambda_{m}=m(n+m-2)$ for $m=0,1,2,\dots$. The eigenspace
$H_m$ of $\lambda_m$ consists of the spherical harmonics of degree $m$ which are in turn
the restriction to $S^{n-1}$ of homogeneous harmonic polynomials of degree $m$ in $\re^{n}$.
Hence we have an orthogonal decomposition
\begin{equation*}
L^{2}(S^{n-1})=\bigoplus_{m\geq 0}H_{m}. 
\end{equation*}

Using this we can perform a similar orthogonal decomposition for the vertical Laplacian on $SM$, 
\[L^{2}(SM)=\bigoplus_{m\geq 0}H_{m}(SM)\]
which on each fibre over $M$ is just the decomposition in $S^{n-1}$.
We set $\Omega_m:=H_{m}(SM)\cap C^{\infty}(SM)$. 
 Then a function $u$ is in $\Omega_m$ if and only if
$\Delta u=m(m+n-2)u$ (for details on this see \cite{GK2,DS}). If $u \in L^2(SM)$, this decomposition will be written as 
$$
u = \sum_{m=0}^{\infty} u_m, \qquad u_m \in H_m(SM).
$$
Note that by orthogonality we have the identities (recall that all norms are $L^2$) 
\begin{align*}
\norm{u}^2 &= \sum_{m=0}^{\infty} \norm{u_m}^2, \\
\norm{\vd u}^2 &= \sum_{m=0}^{\infty} \norm{\vd u_m}^2 = \sum_{m=0}^{\infty} m(m+n-2) \norm{u_m}^2.
\end{align*}

\vspace{10pt}

\noindent {\bf Decomposition of $X$.} The geodesic vector field behaves nicely with respect to the decomposition into fibrewise
spherical harmonics: it maps $\Omega_{m}$ into $\Omega_{m-1}\oplus\Omega_{m+1}$ \cite[Proposition 3.2]{GK2}. Hence on $\Omega_{m}$ we can write
\[X=X_{-}+X_{+}\] 
where $X_{-}:\Omega_{m}\to \Omega_{m-1}$ and $X_{+}:\Omega_{m}\to\Omega_{m+1}$.
By \cite[Proposition 3.7]{GK2} the operator $X_{+}$ is overdetermined elliptic (i.e. it has injective principal symbol). One can gain insight into why the decomposition $X=X_{-}+X_{+}$ holds a follows. Fix $x\in M$ and consider local coordinates which are geodesic at $x$ (i.e. all Christoffel symbols vanish at $x$).
Then $Xu(x,v)=v^{i}\frac{\partial u}{\partial x^{i}}$. We now use the following basic fact about spherical harmonics: the product of a spherical harmonic of degree $m$ with a spherical harmonic of degree one decomposes as the sum of a spherical harmonics of degree $m-1$ and $m+1$. Since the $v^i$ have degree one, this explains why $X$ maps $\Omega_m$ to $\Omega_{m-1} \oplus \Omega_{m+1}$.

Next we give some basic properties of $X_{\pm}$.

\begin{Lemma} $X_{+}^*=-X_{-}$.
\end{Lemma}
\begin{proof} Since $X$ preserves the volume of $SM$, $X^*=-X$ and hence for
$f\in \Omega_m$ and $h\in\Omega_{m+1}$ we have
\[-(f,X_{-}h)=-(f,Xh)=(Xf,h)=(X_{+}f+X_{-}f,h)=(X_{+}f,h). \qedhere \]
\end{proof}

The self-adjoint operator $X_{-}X_{+}$ is elliptic and its kernel coincides with the kernel of $X_{+}$. (This is also considered in \cite{DS}.)

\begin{Lemma} We have an orthogonal decomposition $\Omega_m=X_{+}\Omega_{m-1}\oplus\mbox{\rm Ker}\,X_{-}$.
\label{lemma:decomp}

\end{Lemma}

\begin{proof} Given $f\in \Omega_{m}$ consider $X_{-}f\in\Omega_{m-1}$.
Clearly $X_{-}f$ is orthogonal to the kernel of $X_{-}X_{+}$ since the kernel of $X_{-}X_{+}$ coincides with the kernel of $X_{+}$. Hence by ellipticity there is a smooth solution $h\in\Omega_{m-1}$ such that
$X_{-}X_{+}h=X_{-}f$. Define $q:=f-X_{+}h$, then clearly $X_{-}q=0$.
\end{proof}

\begin{Lemma} \label{lemma:for1}
Given $u\in\Omega_m$ we have
\begin{align*}
[X_{+},\Delta]u&=-(2m+n-1)X_{+}u,\\
[X_{-},\Delta]u&=(2m+n-3)X_{-}u.
\end{align*}
\end{Lemma}
\begin{proof} This is obvious once we know that $X_{\pm}:\Omega_{m}\to\Omega_{m\pm 1}$.
\end{proof}

\noindent {\bf Identification with trace free symmetric tensors.} There is an identification of $\Omega_m$ with the smooth {\it trace free} symmetric tensor fields of degree $m$ on $M$ which we denote by $\Theta_m$ \cite{DS,GK2}.
More precisely, as in \cite{DS} let $\lambda:C^{\infty}(S^*\tau')\to C^{\infty}(SM)$ be the map which takes a symmetric $m$-tensor $f$ and maps it into the function $SM\ni (x,v)\mapsto f_{x}(v,\dots,v)$, where $C^{\infty}(S^*\tau')$ is the space of all smooth symmetric covariant tensors.
The map $\lambda$ turns out to be an isomorphism between  $\Theta_m$ and $\Omega_m$.
In fact up to a factor which depends on $m$ and $n$ only it is a linear isometry when the spaces are endowed with the obvious $L^2$-inner products, cf. \cite[Lemma 2.4]{DS} and \cite[Lemma 2.9]{GK2}.
If $\delta$ denotes divergence of tensors, then in \cite[Section 10]{DS}
we find the formula
\[X_{-}u=\frac{m}{n+2m-2}\lambda\delta\lambda^{-1}u,\]
for $u\in\Omega_m$.
The expression for $X_{+}$ in terms of tensors is as follows. If $d$ denotes symmetrized covariant derivative (formal adjoint of $-\delta$) and $p$ denotes orthogonal projection onto $\Theta_{m+1}$ then
\[X_{+}u=\lambda pd\lambda^{-1}u\]
for $u\in\Omega_m$.
In other words, up to $\lambda$, $X_{+}$ is $pd$ and $X_{-}$ is 
$\frac{m}{n+2m-2}\delta$. The operator $X_+$, at least for $m=1$, has many names and is known as the conformal Killing operator, trace-free deformation tensor, or Ahlfors operator.

Under this identification $\mbox{\rm Ker}\,(X_{+}:\Omega_{m}\to \Omega_{m+1})$ consists of the {\it conformal Killing symmetric tensor fields} of rank $m$, a finite dimensional space. The dimension of this space depends only on the conformal class
of the metric. 

\medskip

\noindent {\bf Pestov identity on $\Omega_m$.} We will see in Appendix \ref{sec_2d_case} that in two dimensions, the Pestov identity specialized to functions in $\Omega_m$ is just the Guillemin-Kazhdan energy identity \cite{GK} involving $\eta_+$ and $\eta_-$. We record here a multidimensional version of this fact.

\begin{Proposition} \label{prop_pestov_omegam}
Let $(M,g)$ be a closed Riemannian manifold. If the Pestov identity is applied to functions in $\Omega_m$, one obtains the identity 
\begin{equation*}
(2m+n-3)\norm{X_{-}u}^{2}+\norm{\hd u}^{2}-(R\vd u,\vd u)=(2m+n-1)\norm{X_{+}u}^{2}
\end{equation*}
which is valid for any $u\in\Omega_m$.
\end{Proposition}

For the proof, we need a commutator formula for the geodesic vector field and the vertical Laplacian.

\begin{Lemma} The following commutator formula holds:
\[[X,\Delta]=2\vdiv\hd+(n-1)X.\]
\label{lemma:co-lap}
\end{Lemma}
\begin{proof}
Using Lemma \ref{lemma_basic_commutator_formulas} repeatedly we have 
\begin{align*}
[X,\Delta]&=-X\vdiv\vd+\vdiv\vd X\\
&=-X\vdiv\vd+\vdiv(X\vd+\hd)\\
&=-\vdiv X\vd+\hdiv\vd+\vdiv(X\vd+\hd)\\
&=\hdiv\vd+\vdiv\hd \\
&=2\vdiv\hd+(n-1)X. \qedhere
\end{align*}
\end{proof}

\begin{proof}[Proof of Proposition \ref{prop_pestov_omegam}]
Let $u \in \Omega_m$. We begin with $X\vd u=\vd Xu-\hd u$ and use Lemmas \ref{lemma:for1} and \ref{lemma:co-lap} to derive
\begin{align*}
\norm{X\vd u}^{2}&=\norm{\vd Xu}^{2}-2 \,\mathrm{Re}(\vd Xu,\hd u)+\norm{\hd u}^{2} \\
&=\norm{\vd Xu}^{2}+\mathrm{Re}( Xu,2\vdiv\hd u )+\norm{\hd u}^{2}\\
&=\norm{\vd Xu}^{2}+\mathrm{Re}( Xu,[X,\Delta]u-(n-1)Xu )+\norm{\hd u}^{2}\\
&=\norm{\vd Xu}^{2}-(n-1)\norm{Xu}^{2}+\mathrm{Re}( Xu,[X_{+},\Delta]u+[X_{-},\Delta]u)+\norm{\hd u}^{2}\\
&=\norm{\vd Xu}^{2}-(n-1)\norm{Xu}^{2}-(2m+n-1)\norm{X_{+}u}^{2}+(2m+n-3)\norm{X_{-}u}^{2}\\
&\;\;\;\;\;\;\;\;\;\;\;\;+\norm{\hd u}^{2}.
\end{align*}
The result follows directly from Proposition \ref{prop_pestov} and the calculation above.
\end{proof}

The Pestov identity on $\Omega_m$ immediately implies a vanishing theorem for conformal Killing tensors. The following result is proved also in \cite[Theorem 1.6]{DS} (except for the observation about rank one manifolds).

\begin{Corollary} Let $(M,g)$ be a closed Riemannian manifold of non-positive sectional curvature with
a transitive geodesic flow. Then there are no non-trivial conformal Killing tensors (CKTs).
In particular if $(M,g)$ is a rank one manifold of non-positive sectional curvature, there are no non-trivial
CKTs.
\label{cor:noCKT}
\end{Corollary}

\begin{proof} Consider $u\in\Omega_m$ such that $X_{+}u=0$ and $m\geq 1$. Since the sectional curvature is non-positive, $(R\vd u,\vd u)\leq 0$ and Proposition \ref{prop_pestov_omegam} gives $X_{-}u=0$ and $\hd u=0$. In particular $Xu=X_{-}u+X_{+}u=0$. If the geodesic flow is transitive, this implies
$u=0$.

In \cite[Theorem 3.11]{Ebe0}, P. Eberlein proved that the geodesic flow of a closed rank one manifold of non-positive sectional curvature is transitive, hence these manifolds do not have CKTs.
\end{proof}

\section{$\alpha$-controlled manifolds} \label{sec_alphacontrolled}

From the Pestov identity in Proposition \ref{prop_pestov}, one would like to obtain good lower bounds for $\norm{\vd Xu}^2$. On the other side of the identity, the term $(n-1)\norm{Xu}^2$ is always nonnegative. The next definition is concerned with the remaining terms (this notion was introduced in \cite{PSU3} for two-dimensional manifolds).

\begin{Definition}
Let $\alpha$ be a real number. We say that a closed Riemannian manifold $(M,g)$ is \emph{$\alpha$-controlled} if 
$$
\norm{X Z}^2 - (R Z, Z) \geq \alpha \norm{X Z}^2
$$
for all $Z \in \mathcal Z$.
\end{Definition}
Note that if $Z(x,v)\neq 0$, then the sign of $\langle RZ(x,v),Z(x,v)\rangle$ is the same as the sign of the sectional curvature of the two-plane spanned by $v$ and $Z(x,v)$.

We record the following properties.

\begin{Lemma} \label{lemma_alphacontrolled_properties}
Let $(M,g)$ be a closed Riemannian manifold. Then 
\begin{itemize}
\item 
$(M,g)$ is $0$-controlled if it has no conjugate points,
\item 
$(M,g)$ is $\alpha$-controlled for some $\alpha > 0$ if it is Anosov,
\item 
$(M,g)$ is $1$-controlled iff it has nonpositive sectional curvature.
\end{itemize}
\end{Lemma}
\begin{proof}
$(M,g)$ is $1$-controlled iff $(RZ,Z) \leq 0$ for all $Z \in \mathcal{Z}$, which is equivalent with nonpositive sectional curvature. The fact that manifolds without conjugate points are $0$-controlled will be proved in Proposition \ref{prop:greeneq}.
Similarly, the fact that Anosov manifolds are $\alpha$-controlled for some $\alpha > 0$ will be proved in Theorem \ref{thm:alphacontrol}.
\end{proof}

We conclude this section with a lemma that will allow to get lower bounds by using the $\alpha$-controlled assumption.

\begin{Lemma} \label{lemma_xvu_lowerbound}
If $u \in C^{\infty}(SM)$ and $u = \sum_{l=m}^{\infty} u_l$, then 
$$
\norm{X \vd u}^2 \geq \left\{ \begin{array}{ll} \frac{(m-1)(m+n-2)^2}{m+n-3} \norm{(Xu)_{m-1}}^2 + \frac{m(m+n-1)^2}{m+n-2} \norm{(Xu)_m}^2, & m \geq 2, \\[5pt]
\frac{n^{2}}{n-1}\norm{(Xu)_{1}}^{2}, & m = 1. \end{array} \right.
$$
If $u \in \Omega_m$, we have 
$$
\norm{X \vd u}^2 \geq \left\{ \begin{array}{ll} \frac{(m-1)(m+n-2)^2}{m+n-3} \norm{X_- u}^2 + \frac{m^2(m+n-1)}{m+1} \norm{X_+ u}^2, & m \geq 2, \\[5pt] \frac{n}{2} \norm{X_+ u}^2, & m = 1. \end{array} \right.
$$
\end{Lemma}

Observe that degree zero is irrelevant since $\vd u_{0}=0$.
The proof relies on another lemma:

\begin{Lemma} \label{lemma_nablah_nablav_innerproduct}
If $u \in C^{\infty}(SM)$ and $w_l \in \Omega_l$, then 
$$
(\hd u, \vd w_l) = ( (l+n-2) X_+ u_{l-1} - l X_- u_{l+1}, w_l).
$$
As a consequence, for any $u \in C^{\infty}(SM)$ we have the decomposition 
$$
\hd u = \vd \left[ \sum_{l=1}^{\infty} \left( \frac{1}{l} X_+ u_{l-1} - \frac{1}{l+n-2} X_- u_{l+1} \right) \right] + Z(u)
$$
where $Z(u) \in \mathcal Z$ satisfies $\vdiv \,Z(u) = 0$.
\end{Lemma}

\begin{proof}
By Lemmas \ref{lemma:for1} and \ref{lemma:co-lap}, 
\begin{align*}
(\hd u, \vd w_l) &= - (\vdiv \hd u, w_l) = -\frac{1}{2} ([X, \Delta] u, w_l) + \frac{n-1}{2} (Xu, w_l) \\
 &= -\frac{1}{2} ([X_+, \Delta] u + [X_-, \Delta] u, w_l) + \frac{n-1}{2} (Xu, w_l) \\
 &= -\frac{1}{2} ([X_+, \Delta] u_{l-1} + [X_-, \Delta] u_{l+1}, w_l) + \frac{n-1}{2} (X_+ u_{l-1} + X_- u_{l+1}, w_l) \\
 &= \left( \frac{2l+n-3}{2} X_+ u_{l-1} - \frac{2l+n-1}{2} X_- u_{l+1}, w_l \right) \\
 & \qquad + \frac{n-1}{2} (X_+ u_{l-1} + X_- u_{l+1}, w_l)
\end{align*}
which proves the first claim. For the second one, we note that 
\begin{multline*}
(\hd u, \vd w) = \sum_{l=1}^{\infty} ( (l+n-2) X_+ u_{l-1} - l X_- u_{l+1}, w_l) \\
 = \sum_{l=1}^{\infty} \frac{1}{l(l+n-2)} ( \vd \left[ (l+n-2) X_+ u_{l-1} - l X_- u_{l+1} \right], \vd w_l)
\end{multline*}
so 
$$
(\hd u - \vd \left[ \sum_{l=1}^{\infty} \left( \frac{1}{l} X_+ u_{l-1} - \frac{1}{l+n-2} X_- u_{l+1} \right) \right], \vd w) = 0
$$
for all $w \in C^{\infty}(SM)$.
\end{proof}

\begin{proof}[Proof of Lemma \ref{lemma_xvu_lowerbound}]
Let $u = \sum_{l=m}^{\infty} u_l$ with $m \geq 2$. First note that 
$$
\norm{X \vd u}^2 = \norm{\vd X u - \hd u}^2.
$$
We use the decomposition in Lemma \ref{lemma_nablah_nablav_innerproduct}, which implies that 
\begin{multline*}
\vd Xu - \hd u = \\
 \vd \Bigg[ \left( 1 + \frac{1}{m+n-3} \right) (Xu)_{m-1} + \left( 1 + \frac{1}{m+n-2} \right) (Xu)_m  + \sum_{l=m+1}^{\infty} w_l \Bigg] + Z
\end{multline*}
where $w_l \in \Omega_l$ for $l \geq m+1$ are given by 
$$
w_l = (Xu)_l - \frac{1}{l} X_+ u_{l-1} + \frac{1}{l+n-2} X_- u_{l+1}
$$
and where $Z \in \mathcal Z$ satisfies $\vdiv\,Z = 0$. Taking the $L^2$ norm squared, and noting that the term $\vd(\,\cdot\,)$ is orthogonal to the $\vdiv$-free vector field $Z$, gives 
\begin{multline*}
\norm{X \vd u}^2 = \frac{(m-1)(m+n-2)^2}{m+n-3} \norm{(Xu)_{m-1}}^2 + \frac{m(m+n-1)^2}{m+n-2} \norm{(Xu)_m}^2 \\
 + \sum_{l=m+1}^{\infty} \norm{\vd w_l}^2 + \norm{Z}^2.  
\end{multline*}
The claims for $m=1$ or for $u \in \Omega_m$ are essentially the same.
\end{proof}

\section{Beurling transform} \label{section_beurling}

In this section we will prove Theorem \ref{theoremA}. The main step is the following inequality, where the point is that the constant in the norm estimate is always $\leq 1$ in dimensions $n \geq 4$, is $\leq 1$ in two dimensions unless $m=1$, and is sufficiently close to $1$ in three dimensions for all practical purposes.

\begin{Lemma} \label{prop_beurling_apriori}
Let $(M,g)$ be a closed Riemannian manifold having nonpositive sectional curvature. One has for any $m\geq 1$
$$
\norm{X_- u} \leq D_{n}(m) \norm{X_{+}u}, \qquad u \in \Omega_m,
$$
where
\begin{align*}
D_2(m) &= \left\{ \begin{array}{cl} \sqrt{2}, & m = 1, \\ 1, & m \geq 2, \end{array} \right. \\
D_3(m) &= \left[ 1 + \frac{1}{(m+1)^2(2m-1)} \right]^{1/2}, \\
D_n(m) &\leq 1 \ \text{ for $n \geq 4$.}
\end{align*}

\end{Lemma}

Recall that the Beurling transform can be defined on any manifold $(M,g)$ that is free of nontrivial conformal Killing tensors. We know that this holds on any surface of genus $\geq 2$, and more generally on any manifold whose conformal class contains a negatively curved metric \cite{DS} or a rank one metric of non-positive curvature (cf. Corollary \ref{cor:noCKT}).

If $(M,g)$ has no nontrivial conformal Killing tensors, the operator $X_{-}:\Omega_{k}\to\Omega_{k-1}$ is surjective for all $k\geq 2$. Hence given $k\geq 0$ and $f_k\in\Omega_k$ there is a unique function $f_{k+2}\in\Omega_{k+2}$ orthogonal to $\mbox{\rm Ker} (X_{-})$ such that $X_{-}f_{k+2}=-X_{+}f_{k}$. We defined the Beurling transform to be the map 
$$
B: \Omega_k \to \Omega_{k+2}, \ \ f_k \mapsto f_{k+2}.
$$

The next lemma shows that a norm estimate relating $X_-$ and $X_+$ implies a bound for the Beurling transform whenever it is defined.

\begin{Lemma} \label{lemma_beurling_from_apriori}
Let $(M,g)$ be a closed Riemannian manifold, let $m \geq 0$, and assume that for some $A > 0$ one has 
$$
\norm{X_- u} \leq A \norm{X_+ u}, \qquad u \in \Omega_{m+1}.
$$
If additionally $(M,g)$ has no conformal Killing $(m+1)$-tensors, then the Beurling transform is well defined $\Omega_m \to \Omega_{m+2}$ and 
$$
\norm{Bf} \leq A \norm{f}, \qquad f \in \Omega_m.
$$
\end{Lemma}
\begin{proof}
Let $f \in \Omega_m$ and $u = Bf$, so that $X_- u = -X_+ f$ and $u \perp \mathrm{Ker}(X_-)$. Then by Lemma \ref{lemma:decomp}, $u = X_+ v$ for some $v \in \Omega_{m+1}$. We have 
\begin{align*}
\norm{u}^2 = (u, X_+ v) = -(X_- u, v) = (X_+ f, v) = -(f,X_- v).
\end{align*}
By Cauchy-Schwarz and by the norm estimate in the statement, we have 
$$
\norm{u}^2 \leq \norm{f} \,\norm{X_- v} \leq A \norm{f} \,\norm{X_+ v} = A \norm{f} \,\norm{u}.
$$
This shows that $\norm{u} \leq A \norm{f}$ as required.
\end{proof}

\begin{Lemma} \label{lemma_beurling_n_neq_three}
Let $(M,g)$ be a closed Riemannian manifold, and let $m \geq 2$. If $(M,g)$ is $\alpha$-controlled with  
$$
\alpha > \frac{(m-2)(m+n-3)}{(m-1)(m+n-2)},
$$
then 
\begin{equation} \label{beurling_alpha_claim1}
\norm{X_- u}^2 \leq C_{\alpha}^2 \norm{X_+ u}^2, \qquad u \in  \Omega_m,
\end{equation}
where $C_{\alpha}$ is the positive constant satisfying 
$$
C_{\alpha}^2 = -\frac{n-1 - (m+1)(m+n-1) + \alpha \frac{m^2(m+n-1)}{m+1}}{n-1 - (m-1)(m+n-3) + \alpha \frac{(m-1)(m+n-2)^2}{m+n-3}}.
$$
Moreover, if $(M,g)$ is $\alpha$-controlled for some $\alpha \geq 0$ and if $m = 1$, one has 
\begin{equation} \label{beurling_alpha_claim15}
\norm{X_- u}^2 \leq \frac{n+1-\alpha\frac{n}{2}}{n-1} \norm{X_+ u}^2, \qquad u \in \Omega_1.
\end{equation}
In particular, if $(M,g)$ has nonpositive sectional curvature and if $n \neq 3$, then one has 
\begin{equation} \label{beurling_alpha_claim2}
\norm{X_{-}u}^2 \leq \left\{ \begin{array}{cl}\norm{X_+ u}^2, & u \in  \Omega_m \text{ with } m \geq 2, \\
\frac{n+2}{2n-2} \norm{X_+ u}^2, & u \in  \Omega_1. \end{array} \right.
\end{equation}
For $n=3$ and $(M,g)$ of nonpositive sectional curvature one has
\begin{equation}\label{beurling_n=3}
\norm{X_{-}u}^{2}\leq \left(1+\frac{1}{(m+1)^{2}(2m-1)}\right)\norm{X_{+}u}^{2}
\end{equation}
for $u\in\Omega_m$ and $m\geq 1$.
\end{Lemma}
\begin{proof}
Let $u \in \Omega_m$ with $m \geq 2$. The Pestov identity in Proposition \ref{prop_pestov} and the $\alpha$-controlled assumption imply that 
\begin{align*}
\norm{\vd X u}^2 &= \norm{X\vd u}^2-(R\,\vd u,\vd u) + (n-1)\norm{Xu}^2 \\
 &\geq \alpha \norm{X \vd u}^2 + (n-1)\norm{Xu}^2.
\end{align*}
Since $Xu = X_+ u + X_- u$ and $u \in \Omega_m$, orthogonality implies that 
\begin{align*}
\norm{Xu}^2 &= \norm{X_+ u}^2 + \norm{X_- u}^2, \\
\norm{\vd Xu}^2 &= (m+1)(m+n-1) \norm{X_+ u}^2 + (m-1)(m+n-3) \norm{X_- u}^2.
\end{align*}
Also, Lemma \ref{lemma_xvu_lowerbound} applied to $u \in \Omega_m$ yields 
$$
\norm{X \vd u}^2 \geq \frac{(m-1)(m+n-2)^2}{m+n-3} \norm{X_- u}^2 + \frac{m^2(m+n-1)}{m+1} \norm{X_+ u}^2.
$$
Collecting these facts gives 
\begin{multline*}
\left[ n-1 - (m-1)(m+n-3) + \alpha \frac{(m-1)(m+n-2)^2}{m+n-3} \right] \norm{X_- u}^2 \\
 \leq \left[ (m+1)(m+n-1) - (n-1) - \alpha \frac{m^2(m+n-1)}{m+1} \right] \norm{X_+ u}^2.
\end{multline*}
The constant in brackets on the right is always positive since $\alpha \leq 1$, and the constant in brackets on the left is positive if $\alpha$ satisfies the condition in the statement. This proves \eqref{beurling_alpha_claim1}.

When $m=1$ the Pestov identity and Lemma \ref{lemma_xvu_lowerbound} yield, by the argument above, 
$$
2n \norm{X_+ u}^2 \geq \alpha \frac{n}{2} \norm{X_+ u}^2 + (n-1)(\norm{X_- u}^2 + \norm{X_+ u}^2)
$$
which implies \eqref{beurling_alpha_claim15}.

If $(M,g)$ has nonpositive sectional curvature, then one can take $\alpha = 1$. Computing $C_{\alpha}$ in this case gives, for $u \in \Omega_m$ with $m \geq 2$, 
$$
\norm{X_- u}^2 \leq \frac{m(m+n-3)(2m+n)}{(m+1)(m+n-2)(2m+n-4)} \norm{X_+ u}^2.
$$
Simplifying the constant further implies that 
$$
\norm{X_- u}^2 \leq \frac{2m^3 + (3n-6)m^2 + n(n-3)m}{2m^3 + (3n-6)m^2 + n(n-3)m + (n-2)(n-4)} \norm{X_+ u}^2.
$$
The constant is always $=1$ if $n = 2$ or $n=4$, is $< 1$ if $n \geq 5$, but is $> 1$ when $n=3$. Thus we have proved \eqref{beurling_alpha_claim2} for $n \neq 3$ and $m \geq 2$, and the case $m=1$ follows by \eqref{beurling_alpha_claim15}. The inequality (\ref{beurling_n=3}) for $n=3$ follows in a similar manner.
\end{proof}

There is an immediate consequence of the previous lemma to the existence of conformal Killing tensors:

\begin{Lemma} \label{lemma_absence_ckt}
Let $(M,g)$ be a closed manifold with transitive geodesic flow. If $m \geq 2$ and $(M,g)$ is $\alpha$-controlled with  
$$
\alpha > \frac{(m-2)(m+n-3)}{(m-1)(m+n-2)},
$$
or if $m=1$ and $(M,g)$ is $0$-controlled, then $(M,g)$ has no nontrivial conformal Killing $m$-tensors (and the same is true for any manifold conformal to $(M,g)$).
\end{Lemma}
\begin{proof}
By Lemma \ref{lemma_beurling_n_neq_three}, the conditions imply that any solution $u \in \Omega_m$ of $X_+ u = 0$ also satisfies $X_- u = 0$. Thus $Xu = 0$, and transitivity of the geodesic flow implies that $u$ is a constant. Since $u \in \Omega_m$, we get $u = 0$.
\end{proof}

It is now easy to give the proofs of Lemma \ref{prop_beurling_apriori} and Theorem \ref{theoremA}.

\begin{proof}[Proof of Lemma \ref{prop_beurling_apriori}]
This result follows directly from Lemma \ref{lemma_beurling_n_neq_three}.
\end{proof}

\begin{proof}[Proof of Theorem \ref{theoremA}.]
Let $(M,g)$ be a closed manifold without conformal Killing tensors and with non-positive sectional curvature. By Lemmas \ref{prop_beurling_apriori} and \ref{lemma_beurling_from_apriori} we have 
$$
\norm{Bf}_{L^2} \leq C_{n}(k)\norm{f}_{L^2}, \qquad f \in \Omega_k, \ \ k \geq 0
$$
where $C_n(k):=D_{n}(k+1)$ and coincides with the definition in the statement of Theorem \ref{theoremA}.
Let $w$ be as in Theorem \ref{theoremA}. Since $w_{m_0+2k} = B^k f$, the estimate on Fourier coefficients follows immediately from the fact that the Beurling transform satisfies the inequality above. If $\eps > 0$ we have 
$$
\left( \sum \br{k}^{-1-2\eps} \norm{w_k}^2 \right)^{1/2} \leq \left( \sum \br{k}^{-1-2\eps} \right)^{1/2}A \norm{f} = C_{\eps} \norm{f},
$$
where $A$ is any constant such that $A_{n}(k)\leq A$ for all $n$ and $k$.
This shows that $w \in L^2_x H^{-1/2-\eps}_{v\phantom{\theta}}$.
\end{proof}

\section{Symplectic cocycles, Green solutions and terminator values}
\label{sec_betaconjugate}

In this section we wish to give a characterization of the Anosov condition that involves a very simple one parameter family of symplectic cocycles over the geodesic flow. The characterization will not require a perturbation of the underlying metric and it will be in terms of a critical value related to conjugate points of the 1-parameter family of cocycles. The results here are generalizations to arbitrary dimensions of the the results in \cite[Section 7]{PSU3}.

Let $(M,g)$ be a closed Riemannian manifold of dimension $n$ and let $\phi_t:SM\to SM$ denote the geodesic flow acting on the unit sphere bundle $SM$. Given $(x,v)\in SM$, we let $R(x,v):\{v\}^{\perp}\to \{v\}^{\perp}$ denote the symmetric linear map defined by $R(x,v)w=R_{x}(w,v)v$, where $R$ is the Riemann curvature tensor.

Let $W(x,v)$ denote the kernel of the contact 1-form $\alpha$ at $(x,v)$. Using the vertical and horizontal splitting we have $W(x,v)=\mathcal H(x,v)\oplus \mathcal V(x,v)$ and each $\mathcal H(x,v)$ and $\mathcal V(x,v)$ can be identified with $\{v\}^{\perp}\subset T_{x}M$. Given $\xi\in W(x,v)$ we write it as $\xi=(\xi_{\tt{h}},\xi_{\tt{v}})$, in terms of its horizontal and vertical components, where  $\xi_{\tt{h}},\xi_{\tt{v}}\in\{v\}^{\perp}$, and we consider the unique solution $J_{\xi}$ to the Jacobi equation
\[\ddot{J}+R(\phi_{t}(x,v))J=0\]
with initial conditions $(J(0),\dot{J}(0))=\xi=(\xi_{\tt{h}},\xi_{\tt{v}})$.
(Here and in what follows the dot denotes covariant derivative of the relevant section.)
In the horizontal and vertical splitting we can write the differential of the geodesic flow as
\[d\phi_{t}(\xi)=(J_{\xi}(t),\dot{J}_{\xi}(t)).\]
For details of this we refer to \cite{Pa}. The linear maps $d\phi_{t}:W(x,v)\to W(\phi_{t}(x,v))$ define a symplectic cocycle over the geodesic flow with respect to the symplectic form $\omega:=-d\alpha|_{W}$. We can now embed the derivative cocycle into a 1-parameter family by considering for each $\beta\in\re$, the $\beta$-Jacobi equation:
\begin{equation}
\ddot{J}+\beta R(\phi_{t}(x,v))J=0.
\label{eq:beta}
\end{equation}
This also defines a symplectic cocycle (on the symplectic vector bundle $(W,\omega)$ over $SM$)
\[\Psi_{t}^{\beta}:W(x,v)\to W(\phi_{t}(x,v))\]
by setting
\[\Psi_{t}^{\beta}(\xi)=(J^{\beta}_{\xi}(t),\dot{J}_{\xi}^{\beta}(t))\]
where $J_{\xi}^{\beta}$ is the unique solution to (\ref{eq:beta}) with initial conditions  $\xi=(J_{\xi}^{\beta}(0),\dot{J}_{\xi}^{\beta}(0))$.
Clearly $\Psi^{1}_{t}=d\phi_{t}$. The cocycle $\Psi_{t}^{\beta}$ is generated by the following 1-parameter family of infinitesimal
generators:
\begin{equation}
A_{\beta}(x,v):=\left(
  \begin{array}{ c c }
     0 & \text{Id} \\
     -\beta R(x,v) & 0
  \end{array} \right)
\label{eq:Abeta}
\end{equation}
and since $R(x,v)$ is a symmetric linear map, it is immediate that $\Psi_{t}^{\beta}$ is symplectic (see \cite{Ka} for information on cocycles over dynamical systems). In this section we shall study this family of cocycles putting emphasis on two properties: absence of conjugate points and hyperbolicity.
For completeness we first give the following two definitions.

\begin{Definition} The cocycle $\Psi_{t}^{\beta}$ is free of conjugate points if any non-trivial
solution of the $\beta$-Jacobi equation $\ddot{J}+\beta R(\phi_{t}(x,v))J=0$ with $J(0)=0$ vanishes only at $t=0$.
\end{Definition}

\begin{Definition} The cocycle $\Psi_{t}^{\beta}$ is said to be hyperbolic if there
is a continuous invariant splitting
$W=E_{\beta}^{u}\oplus E_{\beta}^{s}$, and constants $C>0$ and $0<\rho<1<\eta$ such that 
for all $t>0$ we have
\[\|\Psi^{\beta}_{-t}|_{E^{u}}\|\leq C\,\eta^{-t}\;\;\;\;\mbox{\rm
and}\;\;\;\|\Psi^{\beta}_{t}|_{E^{s}}\|\leq C\,\rho^{t}.\]
\label{def:hyp}
\end{Definition}

In order to simplify the notation we will often drop the subscript $\beta$ in $E_{\beta}^{s,u}$ hoping that this will not cause confusion.

\begin{Remark}{\rm It is well known that the bundles $E^s$ and $E^u$ are $(n-1)$-dimensional and Lagrangian. For the purposes of Definition \ref{def:hyp} one could use any norm on $W$ since they are all uniformly equivalent due to the compactness of $SM$ (the constants $C$, $\eta$ and $\rho$ would be different though). There is however an obvious choice of inner product on $W=\mathcal H\oplus\mathcal V$: on each $\mathcal H$ and $\mathcal V$ we have the inner product induced by $g$ on $\{v\}^{\perp}$; this is the same as the restriction of the Sasaki metric on $T_{(x,v)}SM$ to $W(x,v)$.}
\end{Remark}

Of course, saying that $\Psi^{1}_{t}$ is hyperbolic is the same as saying that $(M,g)$ is an Anosov manifold. The two properties are related by the following:

\begin{Theorem} If $\Psi^{\beta}_{t}$ is hyperbolic then $E_{\beta}^s$ and $E_{\beta}^u$ are transversal
to $\mathcal V$ and $\Psi^{\beta}_{t}$ is free of conjugate points.
\label{thm:kling}
\end{Theorem}

\begin{proof} For $\beta=1$ this is exactly the content of Klingenberg's theorem mentioned
in the introduction \cite{K}. For arbitrary $\beta$ this can proved, for example, using the results in \cite{CGIP}
as we now explain.  Let $\Lambda(SM)$ denote the bundle of Lagrangian subspaces in $W$. The subbundles $\mathcal V$, $E^{u,s}$ are all Lagrangian sections of this bundle, but with $E^{s,u}$ only continuous. The key property of $\Psi_{t}^{\beta}$ is that it is {\it optical or positively twisted}
with respect to $\mathcal V$. This means that for any Lagrangian subspace $\lambda$ in $W(x,v)$ intersecting $\mathcal V(x,v)$ non-trivially, the form
\[(\xi,\eta)\mapsto -\omega\left(\xi,\frac{D}{dt}|_{t=0}(\Psi_{t}^{\beta}\eta)\right)\]
restricted to $\lambda\cap \mathcal V(x,v)$ is positive definite.  Here the term $\frac{D}{dt}$ indicates that along a curve
$\xi(t)=(\xi_{\tt{h}}(t),\xi_{\tt{v}}(t))$ we just take covariant derivatives of each one of the components $\xi_{\tt{h}}(t)$, $\xi_{\tt{v}}(t)$.
It is very well-known that the geodesic flow is optical with respect to the vertical distribution, and we can now quickly check that the same is
true for $\Psi_{t}^{\beta}$. Indeed, take $\eta\in V(x,v)$, then using the $\beta$-Jacobi equation
\[\frac{D}{dt}|_{t=0}(\Psi_{t}^{\beta}\eta)=(\dot{J}_{\eta}^{\beta}(0),-\beta\,R(x,v)J_{\eta}^{\beta}(0)).\]
Since $\eta$ is vertical
\[(\dot{J}_{\eta}^{\beta}(0),-\beta\,R(x,v)J_{\eta}^{\beta}(0))=(\eta_{\tt{v}},0).\]
Thus
\[-\omega\left(\eta,\frac{D}{dt}|_{t=0}(\Psi_{t}^{\beta}\eta)\right)=|\eta_{\tt{v}}|^{2}\]
so we have the desired positive twisting with respect to the vertical distribution.
We are now in good shape to apply Theorem 4.8 in \cite{CGIP} to either $E^s$ or $E^u$ to conclude that there are no $\beta$-conjugate points and that both $E^{s}$ and $E^{u}$ are transversal to $\mathcal V$. Strictly speaking Theorem 4.8 in \cite{CGIP} is stated for the derivative cocycle of a Hamiltonian flow, but it is plain the proof works just the same for symplectic cocycles as in our context.
\end{proof}

Let us describe now the Green limit solutions when $\Psi_{t}^{\beta}$ is free of conjugate points \cite{Green}. For two dimensions these constructions are due to E.~Hopf \cite{H0}.  We shall follow the elegant and short exposition in \cite{I} to construct our solutions.

Set $E_{T}(x,v):=\Psi_{-T}^{\beta}(\mathcal V(\phi_{T}(x,v))$. Since $d\pi:E_{T}(x,v)\to \{v\}^{\perp}$ is an isomorphism, there exists a linear map $S_{T}(x,v):\{v\}^{\perp}\to\{v\}^{\perp}$ such that $E_{T}$ is the graph of $S_{T}$; in other words, given $w\in\{v\}^{\perp}$, there exists a unique $w'\in\{v\}^{\perp}$ such that $(w,w')\in E_{T}$ and we set $S_{T}w:=w'$. Since the cocycle is symplectic $E_T$ is Lagrangian. This is equivalent to $S_T$ being symmetric.
The claim is that $S_{T}$ has a limit as $T\to\infty$.

Recall that it is possible to give a partial order to the set of symmetric linear maps by declaring that
$A\succ B$ if $A-B$ is positive definite. Then we need to show that $S_{T}$ is monotone and bounded.
For $t<s$, the linear map $S_{s}-S_{t}$ is obviously symmetric and we observe that its signature does not change in the region $0<t<s$ since $\Psi_{t}^{\beta}$ has no conjugate points. 
An elementary estimate shows that $S_{s}=-\frac{1}{s}\text{Id}+O(s^2)$ and hence
$S_{s}-S_{t}$ is positive definite for $0<t<s$. Consequently $S_{T}$ is monotone increasing as $T$ goes to $\infty$. Similarly the signature of $S_{s}-S_{t}$ does not change in the region $t<0<s$ and by the same estimate,
$S_{s}-S_{t}$ is negative definite in this region. Hence $S_T$ is bounded by $S_{t}$ for $t<0$.
We let
\[U^{-}:=\lim_{T\to \infty}S_{T}.\]
The graph of $U^{-}$ determines an invariant Lagrangian subbundle $E^{-}$ (stable bundle) which is in general only measurable. Moreover $U^{-}$ is measurable, bounded and satisfies the Riccati equation
\[\dot{U}+U^2+\beta R=0.\]
These claims are all proved as in the case of $\beta=1$ (geodesic flows).
Similarly, by considering $\Psi_{T}^{\beta}(\mathcal (\phi_{-T}(x,v))$ we obtain an invariant subbundle $E^{+}$ (unstable bundle) and a symmetric map $U^{+}$ also solving the Riccati equation above along geodesics. 
Let us agree that given two symmetric linear maps $A$ and $B$, $A\succeq B$ means that $A-B$ is a non-negative operator.
By construction $U^{+}\succeq U^{-}$ since $S_{t}\succ S_{s}$ for $t<0<s$. We summarize these properties in the following lemma.

\begin{Lemma} Assume $\Psi_{t}^{\beta}$ is free of conjugate points. Then there exist symmetric linear maps $U^{\pm}(x,v):\{v\}^{\perp}\to\{v\}^{\perp}$ such that $(x,v)\mapsto U^{\pm}(x,v)$ are measurable, bounded and $t\mapsto U^{\pm}(\phi_{t}(x,v))$ satisfy the Ricati equation
\[\dot{U}+U^2+\beta R=0.\]
Moreover, $U^{+}-U^{-}$ is a non-negative operator.
\label{lemma:green}
\end{Lemma}

We call the symmetric maps $U^{\pm}$  {\it the Green solutions} and often we shall use a subscript $\beta$
to indicate that they are associated with the cocycle $\Psi_{t}^{\beta}$.

\begin{Theorem} Assume that $\Psi_{t}^{\beta}$ is free of conjugate points. Then
$\Psi_{t}^{\beta}$ is hyperbolic if and only if $E^{+}_{\beta}(x,v)\cap E^{-}_{\beta}(x,v)=\{0\}$
for all $(x,v)\in SM$ (equivalently $U^{+}\succ U^{-}$ for all $(x,v)\in SM$).
\label{thm:eberlein}
\end{Theorem}

\begin{proof} For $\beta=1$ this was proved by Eberlein in \cite{Ebe}. To prove the theorem for arbitrary $\beta$ we shall make use of Theorem 0.2
in \cite{CI}. When applied to our situation, it says that
$\Psi_{t}^{\beta}$ is hyperbolic if and only if
\begin{equation}
\sup_{t\in\re}\,\norm{\Psi_{t}^{\beta}(\xi)}=+\infty\;\;\;\mbox{\rm for\;all}\;\xi\in W,\;\xi\neq 0.
\label{eq:quasihyp}
\end{equation}

We shall also need the following proposition:

\begin{Proposition} Assume $\Psi_{t}^{\beta}$ is free of conjugate points
and let $\gamma$ be a unit speed geodesic.
Given $A>0$ there exists $T=T(A,\gamma)$ such that for any solution
$w$ of $\ddot{w}+\beta R(\gamma,\dot{\gamma})w=0$ with $w(0)=0$ we have
\[|w(s)|\geq A |\dot{w}(0)|\]
for all $s\geq T$.
\label{prop:eb2.9}
\end{Proposition}

\begin{proof} The proof of this is exactly like the proof of Proposition 2.9 in \cite{Ebe} and hence we omit it.
\end{proof}

Suppose now we have a solution $J$ to the $\beta$-Jacobi equation
$\ddot{J}+\beta RJ=0$ that is bounded in forward time, i.e., there is
$C$ such that $|J(t)|\leq C$ for all $t\geq 0$.
We claim that $U^{-}_{\beta}(x,v)J(0)=\dot{J}(0)$. For $r>0$, consider the unique matrix
solution $Y_{r}$ of the $\beta$-Jacobi equation with $Y_{r}(r)=0$
and $Y_{r}(0)=\text{Id}$.  The Lagrangian bundle $\Psi_{t-r}^{\beta}(\mathcal V(\phi_{r}(x,v)))$ does not touch the vertical for $t<r$, hence it gives rise to symmetric linear maps $U_{r}^{-}(x,v,t)$ satisfying the Riccati equation for $t<r$. Moreover, it is easy to check that $\dot{Y}_{r}(t)=U^{-}_{r}(x,v,t)Y_{r}(t)$

Let $w(t):=J(t)-Y_{r}(t)J(0)$. Since $w(0)=0$ we may apply Proposition \ref{prop:eb2.9} to derive for any $A$, the existence of $T$ such that
\[|w(s)|\geq A |\dot{w}(0)|\]
for all $s\geq T$. Consider $r$ large enough so that $r\geq T$. Then
\[C\geq |J(r)|=|w(r)|\geq A|\dot{w}(0)|\geq A|\dot{J}(0)-U^{-}_{r}(x,v,0)J(0)|.\]
Now let $r\to\infty$ to obtain
\[C\geq A|\dot{J}(0)-U^{-}_{\beta}(x,v)J(0)|\]
and since $A$ is arbitrary the claim $U^{-}_{\beta}(x,v)J(0)=\dot{J}(0)$ follows.

Similarly, if there is a solution $J$ to the $\beta$-Jacobi equation
that is bounded backwards in time we must have $U^{+}_{\beta}(x,v)J(0)=\dot{J}(0)$.
Thus if there is a solution $J$ bounded for all 
times then $U^{+}_{\beta}-U^{-}_{\beta}$ has zero as an eigenvalue along $\gamma$.

Now it is easy to complete the proof of the theorem. Suppose $\Psi_{t}^{\beta}$ is hyperbolic. Then if we consider a solution of the $\beta$-Jacobi equation corresponding to the stable bundle, it must bounded forward in time by definition
of hyperbolicity and hence by the above the graph of $U_{\beta}^{-}$ must be
$E^s$. This implies $E^{-}=E^{s}$. Similarly the graph of $U^{+}_{\beta} $ is
$E^u$ and $E^{+}=E^{u}$.  Since $E^s$ and $E^u$ are transversal it follows that $E^{+}$ and $E^{-}$ are also transversal.

Suppose now $E^{+}$ and $E^{-}$ are transversal everywhere.
By the argument above, any non-trivial solution $J$ of the $\beta$-Jacobi equation must be unbounded. Since
\[ \norm{\Psi_{t}^{\beta}(\xi)}^2=|J(t)|^2+|\dot{J}(t)|^2,\]
where $J$ is the unique solution to the $\beta$-Jacobi equation with
$(J(0),\dot{J}(0))=\xi$, it follows that (\ref{eq:quasihyp}) holds
and hence $\Psi_{t}^{\beta}$ is hyperbolic.
\end{proof}

Below we will find convenient to use the following well-known comparison lemma (cf. \cite[p. 340]{Re}):

\begin{Lemma} Let $U_{i}(t)$, $i=0,1$ be solutions of the matrix initial value problems
\[\dot{U_{i}}+U_{i}^2+R_{i}(t)=0,\;\;U_{i}(0)=w_i,\;\;i=0,1\]
with $R_{i}$ symmetric for $i=1,2$.
Suppose $w_1\succeq w_0$, $R_{1}(t)\preceq R_{0}(t)$ for $t\in [0,t_{0}]$, and $U_{0}(t_0)$ is defined.
Then $U_{1}(t)\succeq U_{0}(t)$ for $t\in [0,t_0]$.
\label{lemma:comp}
\end{Lemma}

\begin{Theorem} Let $\beta_0> 0$.
If $\Psi_{t}^{\beta_{0}}$ is free of conjugate points,
then for any $\beta\in [0,\beta_0]$, $\Psi_{t}^{\beta}$ is also free of conjugate points.
If $\Psi_{t}^{\beta_{0}}$ is hyperbolic, then for any $\beta\in (0,\beta_0]$, $\Psi_{t}^{\beta}$ is also hyperbolic.
\end{Theorem}

\begin{proof} Let $U^{\pm}_{\beta_{0}}$ be the Green solutions associated with
$\Psi_{t}^{\beta_{0}}$. Given $a\in [0,1]$ we have
\[X(aU^{\pm}_{\beta_{0}})+(aU^{\pm}_{\beta_{0}})^2+a\beta_{0}R=(U^{\pm}_{\beta_0})^2a(a-1)\preceq 0.\]
This already implies that the cocycle $\Psi_{t}^{a\beta_0}$ is free of conjugate points.
Indeed, let $q^{\pm}:=a\beta_{0}R-(U^{\pm}_{\beta_{0}})^2a(a-1)$. Then
\[X(aU^{\pm}_{\beta_{0}})+(aU^{\pm}_{\beta_{0}})^2+q^{\pm}=0\]
and $q^{\pm}\succeq a\beta_{0}R$. Lemma \ref{lemma:comp} implies that the cocycle $\Psi_t^{a\beta_{0}}$ is free of conjugate points. Moreover, it also implies that
\[U^{+}_{a\beta_{0}, r}(x,v,t)\succeq aU^{+}_{\beta_{0}}(x,v,t)\]
for all $t>-r$. By letting $r\to\infty$ we derive
\[U^{+}_{a\beta_{0}}\succeq aU^{+}_{\beta_{0}}\]
and similarly
\[aU^{-}_{\beta_{0}}\succeq U^{-}_{a\beta_0}.\]
Putting everything together we have
\begin{equation}
U^{+}_{a\beta_{0}}\succeq aU^{+}_{\beta_{0}}\succeq aU^{-}_{\beta_{0}}\succeq U^{-}_{a\beta_0}.
\label{eq:chainineq}
\end{equation}

Suppose now that $\Psi_{t}^{\beta_0}$ is hyperbolic. Then by Theorem \ref{thm:kling}, $\Psi_{t}^{\beta_0}$ is free of conjugate points and by Theorem \ref{thm:eberlein}
$U^{+}_{\beta_{0}}\succ U^{-}_{\beta_0}$ everywhere. For $a\in (0,1]$, the chain of inequalities
(\ref{eq:chainineq}) implies that $U^{+}_{a\beta_{0}}\succ U^{-}_{a\beta_0}$ everywhere
and again by Theorem \ref{thm:eberlein}, $\Psi_{t}^{a\beta_0}$ is hyperbolic.
\end{proof}

This theorem motivates the following definition.

\begin{Definition} Let $(M,g)$ be a closed Riemannian manifold.
Let $\beta_{Ter}\in [0,\infty]$ denote the supremum of the values of $\beta\geq 0$ for which
$\Psi_{t}^{\beta}$ is free of conjugate points. We call $\beta_{Ter}$ the terminator
value of the manifold.
\end{Definition}

Observe that $\beta_{Ter}$ could be zero. For instance, this would be the case if $(M,g)$ has positive sectional curvature.
We complement this definition with the following lemma:

\begin{Lemma} The manifold $(M,g)$ is free of $\beta_{Ter}$-conjugate points.
Moreover, $\beta_{Ter}=\infty$ if and only if $(M,g)$ has non-positive sectional curvature.
\label{lemma:kbeta}
\end{Lemma}

\begin{proof} The first claim follows from the following general observation: if $\Psi_{t}^{\beta_0}$ has conjugate points, then there is $\varepsilon>0$ such that for
any $\beta\in (\beta_{0}-\varepsilon,\beta_0+\varepsilon)$ the cocycle $\Psi_{t}^{\beta}$ has conjugate points. This is proved in the same way (in fact it is easier) as the proof that the existence of conjugate points is an open condition on the metric $g$ in the $C^k$-topology for any $k\geq 2$; see for example \cite[Corollary 1.2]{Ru}.

For the second claim note that obviously sectional curvature $K\leq 0$ implies that $\beta_{Ter}=\infty$. For the converse, suppose that there is a point $x\in M$ and a $2$-plane $\sigma\subset T_{x}M$ such that the sectional curvature $K_{x}(\sigma)>0$. Let $\{v,w\}$ denote an orthonormal basis of $\sigma$. Let $\gamma$ denote the geodesic defined by $(x,v)$ and let $e(t)$ be the parallel transport of $w$ along $\gamma$. Consider the 2-plane $\sigma_t$ spanned by $\dot{\gamma}(t)$ and $e(t)$. Since $K_{x}(\sigma)>0$ we can find $r,\delta>0$ such that
$K_{\gamma(t)}(\sigma_{t})\geq \delta$ for $t\in [-r,r]$. Consider a smooth function $f$ defined on $[-r,r]$ such that $f$ vanishes around $r$ and $-r$ and $f(t)=1$ for $t\in[-r/2,r/2]$. Consider the vector field $W(t)=f(t)e(t)$ along $\gamma$. Since there are no $\beta$-conjugate points for all $\beta\geq 0$ we must have
(see for example \cite[p.46]{Pestov} or \cite[Lemma 5.3]{Dairbekov_nonconvex}):
\[\int_{-r}^{r}(|\dot{W}|^2-\beta\langle R(W,\dot{\gamma})\dot{\gamma},W\rangle)\,dt\geq 0.\]
Equivalently
\[\int_{-r}^{r}|\dot{W}|^2\,dt-\beta\int_{-r}^{r}f^2K_{\gamma}(\sigma_t)\,dt\geq 0\]
for all $\beta\geq 0$. And this implies
\[\int_{-r}^{r}|\dot{W}|^2\,dt\geq r\beta\delta\]
for all $\beta\geq 0$ which is clearly absurd.
\end{proof}

We now have the following purely geometric characterization of hyperbolicity (the parameter $\beta$
is always $\geq 0$ in what follows). Recall an orthogonal parallel Jacobi field is a parallel field along a geodesic $\gamma$, orthogonal to $\dot{\gamma}$  and such that it satisfies the Jacobi equation.

\begin{Theorem} The cocycle $\Psi_{t}^\beta$ is hyperbolic if and only if $\beta\in (0,\beta_{Ter})$ and there is no geodesic with a nonzero orthogonal parallel Jacobi field.
\end{Theorem}

\begin{proof} We know that if $\Psi_{t}^{\beta}$ is hyperbolic then $\beta\leq \beta_{Ter}$.
Since hyperbolicity is an open condition we must have $\beta<\beta_{Ter}$. Finally if there
is a geodesic with a nonzero orthogonal parallel Jacobi field $J$, then $\Psi_{t}^{\beta}$ cannot be hyperbolic as such a field is bounded and solves the $\beta$-Jacobi equation for any $\beta$.

Consider $\beta\in (0,\beta_{Ter})$ and assume that $\Psi_t^{\beta}$ is not
hyperbolic. By Theorem \ref{thm:eberlein} there is a geodesic $\gamma$ along which $E^{+}\cap E^{-}\neq \{0\}$. Equivalently there is a non-zero $\beta$-Jacobi field $J$ such that $\dot{J}=U^{-}J=U^{+}J$ along $\gamma$.
Let $a:=\beta/\beta_{Ter}$. Using (\ref{eq:chainineq})
for $\beta_{0}=\beta_{Ter}$ we deduce that along $\gamma$ we must have
\[\dot{J}=U^{+}_{\beta}J=aU^{+}_{\beta_{Ter}}J=aU^{-}_{\beta_{Ter}}J= U^{-}_{\beta}J.\]
Differentiating the equality $U^{+}_{\beta}J=aU^{+}_{\beta_{Ter}}J$ with respect to $t$ and using that $\dot{J}=U^{+}_{\beta}J$ we derive
\[\dot{U}^{+}_{\beta}J+(U^{+}_{\beta})^2 J=a\dot{U}^{+}_{\beta_{Ter}}J+a^2(U^{+}_{\beta_{Ter}})^2 J.\]
Using the Riccati equation we arrive at $(U^{+}_{\beta_{Ter}})^2 J=a(U^{+}_{\beta_{Ter}})^2 J$ and hence $(U^{+}_{\beta_{Ter}})^2 J=0$. Since $U^{+}_{\beta_{Ter}}$ is symmetric
this implies that $U^{+}_{\beta_{Ter}}J=0$ and hence $\dot{J}=0$ which contradicts our hypotheses.
\end{proof}

We can certainly rephrase this theorem using the definition of the rank of a unit vector given in the introduction. The condition of having no geodesic with a nonzero orthogonal parallel Jacobi field is equivalent to saying that every unit vector has rank one.
As an immediate consequence we obtain the following geometric characterization of Anosov manifolds.

\begin{Corollary} A closed Riemannian manifold $(M,g)$ is Anosov if and only if every unit vector has rank one and $\beta_{Ter}>1$.
\label{corollary:ganosov}
\end{Corollary}

It is interesting to note that the main example in \cite{BBB} has the property of being a non-Anosov closed surface with no conjugate points such that  every vector has rank one. Hence for this example $\beta_{Ter}=1$.

\section{Anosov manifolds are $\alpha$-controlled} \label{section_anosov_alphacontrolled}

Consider $\beta>0$ such that there are no $\beta$-conjugate points and let $U$ be one of the Green solutions satisfying the Riccati equation
\[\dot{U}+U^2+\beta\,R(\phi_{t}(x,v))=0.\] 
Recall that $U$ is measurable, bounded and differentiable along the geodesic flow, see Lemma \ref{lemma:green}.
The next identity is very similar to \cite[Theorem 3.1]{DP2}, but here we give a short self-contained proof.

\begin{Proposition}Let $(M,g)$ be a closed manifold without $\beta$-conjugate points for some $\beta>0$. Then for any $Z\in \mathcal Z$
\begin{equation}
\norm{XZ-UZ}^2=\norm{XZ}^2-\beta (RZ,Z).
\label{eq:green}
\end{equation}
As a consequence if $\beta\in (0,\beta_{Ter}]$, the manifold $(M,g)$ is $\frac{\beta-1}{\beta}$-controlled.

\label{prop:greeneq}
\end{Proposition}

\begin{proof} Let us write
\[\norm{XZ-UZ}^2=\norm{XZ}^{2}-2(XZ,UZ)+\norm{UZ}^2.\]
All we need to show is that
\[2(XZ,UZ)=\norm{UZ}^{2}+\beta(RZ,Z).\]
For this note that $X^*=-X$ and that $U$ is a symmetric operator, hence using the Riccati equation we derive
\begin{align*}
(XZ,UZ)&=-(Z,X(UZ))=-(Z,(XU)Z+U(XZ))\\
&=(Z,U^2Z)+\beta(RZ,Z)-(Z,U(XZ))\\
&=\norm{UZ}^2+\beta(RZ,Z)-(UZ,XZ)
\end{align*}
and (\ref{eq:green}) is proved.
The last statement in the proposition follows from the fact that $\norm{XZ}^2-\beta (RZ,Z)\geq 0$ is equivalent to $(M,g)$ being $\frac{\beta-1}{\beta}$-controlled.
\end{proof}

\begin{Theorem} Let $(M,g)$ be an Anosov manifold. Then there exists $\alpha>0$ such that
\[\norm{XZ}^2-(RZ,Z)\geq \alpha(\norm{XZ}^2+\norm{Z}^2)\]
for all $Z\in\mathcal Z$.
\label{thm:alphacontrol}
\end{Theorem}
\begin{proof} Since the geodesic flow is Anosov we have continuous stable and unstable bundles. It is well known that these
are invariant, Lagrangian and contained in the kernel of the contact form. Moreover, by Theorem \ref{thm:kling} the subbundles are transversal to the vertical subbundle.
Thus we have two continuous symmetric maps $U^{+}$ and $U^{-}$ satisfying the Riccati equation 
with the property that the linear map $U^{+}(x,v)-U^{-}(x,v)$ is positive definite for every $(x,v)\in SM$.

Let $A:=XZ-U^{-}Z$ and $B:=XZ-U^{+}Z$. Using equation (\ref{eq:green}) we see that $\|A\|=\|B\|$.
Solving for $Z$  we obtain
\[Z=(U^{+}-U^{-})^{-1}(A-B).\]
Since $U^{\pm}$ are continuous, there is a constant $a>0$ independent of $Z$ such that
\[\norm{Z}\leq a\norm{A}.\]
Since $XZ=A+U^{-}Z$ we see that there is another constant $b>0$ independent of $Z$ for which
\[\norm{XZ}\leq b\norm{A}.\] 
From these inequalities and (\ref{eq:green}) the existence of $\alpha$ easily follows.
\end{proof}

\section{Subelliptic estimate and surjectivity of $I_0^*$} \label{sec_surjectivity_izero}

In this section we will show that the Anosov condition allows to upgrade the Pestov identity in Proposition \ref{prop_pestov} to a subelliptic ($H^1$ to $L^2$) estimate for the operator $\vd X$. From this point on, we will use the notation 
$$
P = \vd X.
$$
Note that $P$ maps $C^{\infty}(SM)$ to $\mathcal{Z} = C^{\infty}(SM, N)$ (see Section \ref{sec_spherebundle} for the definitions). Also recall from the discussion before Proposition \ref{prop_pestov} that uniqueness results for $P$ imply injectivity for the geodesic ray transform $I_0$. The subelliptic estimate may then be viewed as a quantitative injectivity result for $I_0$.

If $E$ is a subspace of $\mDp(SM)$, we write $E_{\diamond}$ for the subspace of those $v \in E$ with $( v, 1) = 0$. 

\begin{Theorem} \label{lemma_p_inequalities}
Let $(M,g)$ be an Anosov manifold. Then 
$$
\norm{u}_{H^1} \lesssim \norm{Pu}_{L^2}, \qquad u \in C^{\infty}_{\diamond}(SM).
$$
More generally, for $u \in H^1(SM)$ with $Pu \in L^2(SM)$ one has 
$$
\norm{u - (u)_{SM}}_{H^1} \lesssim \norm{Pu}_{L^2} 
$$
where $(u)_{SM} = \frac{1}{\mathrm{Vol}(SM)} \int_{SM} u$ is the average of $u$.
\end{Theorem}
\begin{proof}
Let $u \in C^{\infty}_{\diamond}(SM)$. We combine the Pestov identity in Proposition \ref{prop_pestov} with
Theorem \ref{thm:alphacontrol}, which yields $\norm{X\vd u}^2-(R\vd u,\vd u)\geq \alpha(\norm{X\vd u}^2+\norm{\vd u}^2)$ for some $\alpha > 0$ by the Anosov property. Thus we obtain 
\[\norm{Pu}^2\geq (n-1)\norm{Xu}^2+\alpha(\norm{\vd u}^2+\norm{X\vd u}^2).\]
The subelliptic estimate also requires a $\norm{\hd u}^2$ term on the right. This follows from the commutator formula $[X, \vd] = -\hd$, which gives 
\[\norm{\hd u}\leq \norm{Pu}+\norm{X\vd u}.\]
From these facts it follows that 
\[\norm{\hd u}\lesssim \norm{Pu}.\]
Hence 
\[\norm{Pu}^2\gtrsim \norm{\hd u}^2+\norm{\vd u}^2+\norm{Xu}^2=\norm{\nabla_{SM} u}^2.\]
By the Poincar\'e inequality for closed Riemannian manifolds, one also has 
\[\|u\|^2\lesssim \norm{\nabla_{SM} u}^2\]
for all $u\in C^{\infty}_{\diamond}(SM)$. Hence 
\[\norm{u}_{H^1} \lesssim \norm{Pu}\]
for all $u\in C^{\infty}_{\diamond}(SM)$ as desired.

Using the Poincar\'e inequality in the form 
$$
\norm{w-(w)_{SM}}_{L^2} \lesssim \norm{\nabla_{SM} w}, \qquad w \in C^{\infty}(SM),
$$
the above argument also implies the inequality 
$$
\norm{w-(w)_{SM}}_{H^1} \lesssim \norm{Pw}_{L^2}, \qquad w \in C^{\infty}(SM).
$$
Since $P$ has smooth coefficients, this last inequality holds also for any $w \in H^1(SM)$ with $Pw \in L^2(SM)$ by convolution approximation and the Friedrichs lemma \cite[Lemma 17.1.5]{Hor} (one covers $SM$ by coordinate neighborhoods, uses a subordinate partition of unity, and performs the convolution approximations in the coordinate charts).
\end{proof}

We now use the subelliptic estimate to obtain a solvability result for $P^*=X\vdiv$. In \cite{PSU3} this was done by a Hahn-Banach argument in the two-dimensional case. Here we give an alternative Hilbert space argument based on the Riesz representation theorem, and we also give a notion of uniqueness for solutions.

\begin{Lemma} \label{lemma_padjoint_surjective}
Let $(M,g)$ be an Anosov manifold. For any $f \in H^{-1}_{\diamond}(SM)$ there is a solution $h\in L^2(SM,N)$ of the equation 
$$
P^* h = f \quad \text{in } SM.
$$
There is a unique solution satisfying one of the following equivalent conditions:
\begin{enumerate}
\item[(a)]
$h$ is in the range of $P$ acting on $\{ u \in H^1 \,;\, Pu \in L^2 \}$.
\item[(b)]
$h$ is $L^2$-orthogonal to the kernel of $P^*$ on $L^2$.
\item[(c)]
$h$ has minimal $L^2$ norm.
\end{enumerate}
The unique solution $h$ satisfies $\norm{h}_{L^2} \lesssim \norm{f}_{H^{-1}}$.
\end{Lemma}
\begin{proof}
Define the space 
$$
\mathcal{A} = \{ u \in H^1_{\diamond}(SM) \,;\, Pu \in L^2(SM) \}
$$
with inner product 
$$
(u, w)_{\mathcal{A}} = (Pu, Pw)_{L^2}.
$$
This is an inner product space, since $(u,u)_{\mathcal{A}} = 0$ for $u \in \mathcal{A}$ implies $u = 0$ by Theorem \ref{lemma_p_inequalities}. It is also a Hilbert space: if $(u_j)$ is a Cauchy sequence in $\mathcal{A}$ then it is a Cauchy sequence in $H^1$ by Theorem \ref{lemma_p_inequalities}, hence $(u_j)$ converges in $H^1$ to some $u \in H^1$ and also $(P u_j)$ converges to some $w$ in $L^2$. Then $Pu_j \to Pu$ in $H^{-1}$ and also $Pu_j \to w$ in $H^{-1}$, showing that $Pu = w \in L^2$ and $u_j \to u$ in $\mathcal{A}$.

Given $f \in H^{-1}_{\diamond}(SM)$, define the functional 
$$
l: \mathcal{A} \to \C, \ \ l(w) = (w,f).
$$
(The expression on the right is the distributional pairing.) This functional satisfies by Theorem \ref{lemma_p_inequalities} 
$$
\abs{l(w)} \leq \norm{w}_{H^1} \norm{f}_{H^{-1}} \lesssim \norm{w}_{\mathcal{A}} \norm{f}_{H^{-1}}.
$$
Thus $l$ is continuous on $\mathcal{A}$, and by the Riesz representation theorem there is $u \in \mathcal{A}$ satisfying 
$$
l(w) = (w,u)_{\mathcal{A}}, \qquad \norm{u}_{\mathcal{A}} \lesssim \norm{f}_{H^{-1}}.
$$
For any $w \in C^{\infty}_{\diamond}(SM)$ we have 
$$
(w,f) = (w,u)_{\mathcal{A}} = (Pw,Pu) = (w,P^* Pu)
$$
Therefore $h = Pu$ solves $P^* h = f$ in $SM$ in the sense of distributions (since both $f$ and $P^* h$ are orthogonal to constants) and $\norm{h}_{L^2} \lesssim \norm{f}_{H^{-1}}$.

We next observe that $h$ is the unique solution of $P^* h = f$ which is in the range of $P$ acting on $\mathcal{A}$. Indeed, if both $h = Pu$ and $\tilde{h} = P\tilde{u}$ solve the equation, then $P^* P(u-\tilde{u}) = 0$ so $(P(u-\tilde{u}), Pw) = 0$ for smooth $w$. We can use the Friedrichs lemma to approximate $u-\tilde{u}$ by smooth functions in the $\mathcal{A}$ norm as in the proof of Theorem \ref{lemma_p_inequalities}. This shows $\norm{P(u-\tilde{u})}^2 = 0$, so $u = \tilde{u}$ by Theorem \ref{lemma_p_inequalities}. The equivalence of (a), (b), (c) follows since any $\tilde{h} \in L^2$ can be expressed as the orthogonal sum $\tilde{h} = Pu + w$ where $u \in \mathcal{A}$ and $w \in L^2$ with $P^* w = 0$ (this follows since $P: \mathcal{A} \to L^2$ is bounded with closed range and the orthocomplement of $\mathrm{Ran}(P)$ is $\mathrm{Ker}(P^*)$, one can again approximate $u$ by smooth functions in the $\mathcal{A}$ norm).
\end{proof}

Let $\omega_{n-1}$ denote the volume of the $(n-1)$-dimensional unit sphere.
Given $w\in \mDp(SM)$, let $w_0\in \mDp(M)$ be defined by
\[\langle w_0, \psi\rangle :=\frac{1}{\omega_{n-1}}\langle w,\psi\circ\pi\rangle, \qquad \psi \in C^{\infty}(M).\]

We can now prove surjectivity of $I_0^*$. See \cite[Section 1.2]{PSU3} for an explanation why the next result is indeed equivalent to surjectivity of $I_0^*$, and observe that $w \mapsto \left[ \norm{w_0}^2 + \sum_{m=1}^{\infty} \frac{1}{m(m+n-2)} \norm{w_m}^2 \right]^{1/2}$ is an equivalent norm on $L^2_x H^{-1}_v$.

\begin{Theorem} \label{lemma_surjectivity_stronger}
Assume that $(M,g)$ is Anosov. Let $F \in H^{-1}_{\diamond}(SM)$ and $f \in L^2(M)$. There exists $w \in L^2_x H^{-1}_v(SM)$ satisfying 
$$
Xw = F \text{ in } SM, \qquad w_0 = f.
$$
There is a unique such $w$ for which the quantity $\sum_{m=1}^{\infty} \frac{1}{m(m+n-2)} \norm{w_m}^2$ is minimal, and this $w$ satisfies $\norm{w}_{L^2_x H^{-1}_v} \lesssim \norm{F}_{H^{-1}} + \norm{f}_{L^2}$.
\end{Theorem}
\begin{proof}
By Lemma \ref{lemma_padjoint_surjective} there is $h \in L^2(SM,N)$ such that $h = \vd a$ for some $a \in L^2(SM)$ and 
$$
P^* h = F - Xf. 
$$
Then $w = \vdiv\,h + f$ is in $L^2_x H^{-1}_v$ (since $\sum_{m=1}^{\infty} \frac{1}{m(m+n-2)} \norm{w_m}^2 = \norm{\vd a}_{L^2}^2 = \norm{h}^2$), and it satisfies 
$$
Xw = X\vdiv\,h + Xf = F
$$
and $w_0 = f$. To check the last equality it suffices to observe that
\[\langle \vdiv\,h,\psi\circ\pi\rangle=0\]
since $\vd (\psi\circ\pi)=0$ (the function $\psi\circ\pi$ only depends on $x\in M$). The uniqueness of $w$ with the minimality property, together with the norm bounds, follow from the corresponding properties of $h$ in Lemma \ref{lemma_padjoint_surjective}.
\end{proof}

\section{Surjectivity of $I_m^*$} \label{sec_surjectivity_im}

Let $(M,g)$ be a closed Anosov manifold of dimension $n \geq 2$. We wish to prove a quantitative injectivity result for $I_m$ when $m \geq 1$, if $(M,g)$ is Anosov and $\alpha$-controlled for suitable $\alpha$. To this end, we consider the operator 
$$
Q_m: T_{\geq m} C^{\infty}(SM) \to \mathcal Z, \ \ Q_m u = \vd T_{\geq m+1} X u
$$
where
$$
T_{\geq r} u = \sum_{m=r}^{\infty} u_m.
$$
The next theorem is a subelliptic estimate for $Q_m$. It may be viewed as a quantitative injectivity result for $I_m$, and it implies a surjectivity result for $I_m^*$ (see again \cite[Section 1.2]{PSU3} for an explanation).

\begin{Theorem} \label{theorem_qm_subelliptic}
If $(M,g)$ is closed Anosov, if $m \geq 1$, and if $(M,g)$ is $\alpha$-controlled for $\alpha > \alpha_{m,n}$ where 
$$
\alpha_{m,n} = \frac{(m-1)(m+n-2)}{m(m+n-1)},
$$
then we have 
$$
\norm{u}_{H^1} \lesssim \norm{Q_m u}_{L^2}, \qquad u \in T_{\geq m} C^{\infty}(SM).
$$
\end{Theorem}
\begin{proof}
Assume that $u \in T_{\geq m} C^{\infty}(SM)$. The Pestov identity applied to $u$ reads 
$$
\norm{\vd X u}^2 = (n-1) \norm{Xu}^2 + \norm{X \vd u}^2 - (R \vd u, \vd u).
$$
The decomposition 
$$
X u = (Xu)_{m-1} + (Xu)_m + T_{\geq m+1} Xu
$$
implies that 
$$
\norm{Xu}^2 = \norm{(Xu)_{m-1}}^2 + \norm{(Xu)_m}^2 + \norm{T_{\geq m+1} X u}^2
$$
and 
\begin{align*}
\norm{\vd X u}^2 &= (m-1)(m+n-3) \norm{(Xu)_{m-1}}^2 + m(m+n-2) \norm{(Xu)_m}^2 \\
 & \qquad + \norm{\vd T_{\geq m+1} X u}^2.
\end{align*}
Using these facts in the Pestov identity above shows that 
\begin{multline*}
\norm{Q_m u}^2 + (m-1)(m+n-3) \norm{(Xu)_{m-1}}^2 + m(m+n-2) \norm{(Xu)_m}^2 \\
 = (n-1) \left[ \norm{(Xu)_{m-1}}^2 + \norm{(Xu)_m}^2 + \norm{T_{\geq m+1} X u}^2 \right] + \norm{X \vd u}^2 - (R \vd u, \vd u).
\end{multline*}
Applying the $\alpha$-controlled assumption and Lemma \ref{lemma_xvu_lowerbound} yields for $m\geq 2$ 
\begin{multline*}
\norm{Q_m u}^2 \geq \left[ n-1 - (m-1)(m+n-3) + \alpha \frac{(m-1)(m+n-2)^2}{m+n-3} \right] \norm{(Xu)_{m-1}}^2 \\ 
 + \left[ n-1 - m(m+n-2) + \alpha \frac{m(m+n-1)^2}{m+n-2} \right] \norm{(Xu)_m}^2 \\
 + (n-1) \norm{T_{\geq m+1} X u}^2.
\end{multline*}
If $\alpha > \alpha_{m,n}$, the constants in brackets are positive and it follows that 
$$
\norm{Xu} \lesssim \norm{Q_m u}, \qquad u \in T_{\geq m} C^{\infty}(SM).
$$
If $m=1$ we obtain
\[\norm{Q_{1}u}^{2}\geq (n-1)[\norm{(Xu)_{0}}^2+\norm{T_{\geq 2} X u}^2]+\frac{\alpha n^{2}}{n-1}\norm{(Xu)_{1}}^{2}\]
and thus we also have
$$
\norm{Xu} \lesssim \norm{Q_1 u}.
$$
From the definitions of $P$ and $Q_m$ we see that
\[Pu=\sum_{l=1}^{m}\vd(Xu)_{l}+Q_{m}u\]
and hence
\[\norm{Pu}^{2}=\sum_{l=1}^{m}l(l+n-2)\norm{(Xu)_{l}}^{2}+\norm{Q_{m}u}^{2}.\]
(The last two identities hold for $u \in C^{\infty}(SM)$.) The subelliptic estimate for $Q_m$ now follows from the above discussion and the basic subelliptic estimate for $P$ given in Theorem \ref{lemma_p_inequalities}.
\end{proof}

The subelliptic estimate implies a solvability result for $Q_m^* = X T_{\geq m+1} \vdiv$. The proofs are similar to those in Section \ref{sec_surjectivity_izero} (the next results could also be made slightly more precise as in Section \ref{sec_surjectivity_izero}; we omit the details).

\begin{Lemma}
Assume the conditions in Theorem \ref{theorem_qm_subelliptic}. For any $f \in T_{\geq m} H^{-1}(SM)$, there is $H \in L^2(SM,N)$ with $Q_m^* H = f$ and $\norm{H}_{L^2(SM,N)} \lesssim \norm{f}_{H^{-1}(SM)}$.
\end{Lemma}
\begin{proof}
If $f$ is as above, define a linear functional 
$$
l: Q_m(T_{\geq m} C^{\infty}(SM)) \subset L^2(SM,N) \to \C, \ \ l(Q_m u) = (u,f).
$$
Then $\abs{l(Q_m u)} \lesssim \norm{f}_{H^{-1}} \norm{Q_m u}$ by Theorem \ref{theorem_qm_subelliptic}, thus $l$ has a bounded linear extension $\bar{l}: L^2(SM,N) \to \C$ by Hahn-Banach and there is $H \in L^2(SM,N)$ with $\norm{H}_{L^2} \lesssim \norm{f}_{H^{-1}}$ and 
$$
\bar{l}(Z) = (Z,H)_{L^2(SM,N)}, \qquad Z \in L^2(SM,N).
$$
But now for any $u \in C^{\infty}(SM)$ we have 
$$
(u,Q_m^* H) = (Q_mu, H) = (Q_m(u-\sum_{l=0}^{m-1} u_l),H) = (u-\sum_{l=0}^{m-1} u_l,f) = (u,f)
$$
so $Q_m^* H = f$ as required.
\end{proof}

This leads to invariant distributions starting at any $a_m$ with $X_- a_m = 0$:

\begin{Proposition}
Assume the conditions in Theorem \ref{theorem_qm_subelliptic}. Let $a_m \in \Omega_m$ with $X_- a_m = 0$. There is $w \in T_{\geq m} H^{-1}(SM)$ with $Xw = 0$ and $w_m = a_m$.
\end{Proposition}
\begin{proof}
Let $f = -X_+ a_m$, and use the previous lemma to find $H \in L^2(SM,N)$ with $Q_m^* H = f$. Then 
$$
X (T_{\geq m+1} \vdiv H) = -X a_m
$$
and so $Xw = 0$ for $w = a_m + T_{\geq m+1} (\vdiv H)$ where $w \in H^{-1}$.
\end{proof}

\section{Injectivity of $I_{m}$} \label{section_injectivity}

In this section we would like to prove the following injectivity result.

\begin{Theorem} \label{theorem_im_injectivity}
Let $(M,g)$ be closed Anosov, let $m \geq 2$, and let $(M,g)$ be $\alpha$-controlled for $\alpha \geq \alpha_{m,n}$ where 
$$
\alpha_{m,n} = \frac{(m-1)(m+n-2)}{m(m+n-1)}.
$$
If additionally $(M,g)$ has no nontrivial conformal Killing $(m+1)$-tensors (this is always true if $\alpha > \alpha_{m,n}$ by Lemma \ref{lemma_absence_ckt}), then $I_m$ is $s$-injective.
\end{Theorem}

This has two immediate corollaries stated in terms of the terminator value of $(M,g)$.

\begin{Corollary} \label{cor_im_injectivity_beta}
Let $(M,g)$ be a closed Riemannian manifold such that every unit vector has rank one, and let $m \geq 2$.
Suppose in addition that $\beta_{Ter}\geq \frac{m(m+n-1)}{2m+n-2}$ and that there are no non-trivial conformal Killing tensors of order $m+1$ (this is always true if $\beta_{Ter} > \frac{m(m+n-1)}{2m+n-2}$). Then $I_{m}$ is s-injective.
\end{Corollary}
\begin{proof}
Observe first that $\frac{m(m+n-1)}{2m+n-2}>1$ and by Corollary \ref{corollary:ganosov} the manifold $(M,g)$ is Anosov. Moreover, if $\beta_{Ter}\geq  \frac{m(m+n-1)}{2m+n-2}$ then the manifold
is $\frac{(m-1)(m+n-2)}{m(m+n-1)}$-controlled. The corollary now follows from Theorem \ref{theorem_im_injectivity}.
\end{proof}

\begin{Corollary} Let $(M,g)$ be a closed Riemannian manifold such that every unit vector has rank one, and assume that $\beta_{Ter}\geq \frac{2(n+1)}{n+2}$. Suppose in addition that there are no non-trivial conformal Killing $3$-tensors  (this is always true if $\beta_{Ter} > \frac{2(n+1)}{n+2}$).
Then $(M,g)$ is spectrally rigid.
\end{Corollary}
\begin{proof}
Follows from Corollary \ref{cor_im_injectivity_beta} with $m=2$ and \cite{GK}.
\end{proof}

\begin{proof}[Proof of Theorem \ref{theorem_im_injectivity}]
Suppose $f$ is a symmetric $m$-tensor in the kernel of $I_m$. Then $f$ gives rise to a function $f(x,v) = f(v,\ldots,v) \in C^{\infty}(SM)$ that integrates to zero over periodic geodesics and has degree $m$ (meaning that $f = \sum_{l \leq m} f_l$). By the Livsic theorem \cite{dLMM}, there is $a\in C^{\infty}(SM)$ with $Xa=f$.
Let $u=a-\sum_{l \leq m-1}a_l$. Then $Xu$ has degree $m$ and $Q_mu=\vd T_{\geq m+1}Xu=0$.

If we examine the proof of Theorem \ref{theorem_qm_subelliptic} and include some additional positive terms from the proof of Lemma \ref{lemma_xvu_lowerbound}, the following inequality emerges for an $\alpha$-controlled manifold:

\begin{multline*}
\norm{Q_m u}^2 \geq \left[ n-1 - (m-1)(m+n-3) + \alpha \frac{(m-1)(m+n-2)^2}{m+n-3} \right] \norm{(Xu)_{m-1}}^2 \\ 
 + \left[ n-1 - m(m+n-2) + \alpha \frac{m(m+n-1)^2}{m+n-2} \right] \norm{(Xu)_m}^2 \\
 + (n-1) \norm{T_{\geq m+1} X u}^2+\alpha(\sum_{l=m+1}^{\infty}\norm{\vd w_{l}}^{2}+\norm{Z}^{2}).
\end{multline*}
The assumption on $\alpha$ implies that the first constant in brackets is positive and the second is nonnegative. If we use that $Q_{m}u=0$ we derive: $(Xu)_{m-1}=0$, $(Xu)_l = 0$ and $\vd w_{l}=0$ for all $l\geq m+1$, and $Z=0$.
This implies (by the proof of Lemma \ref{lemma_xvu_lowerbound}) 
\begin{align*}
X\vd u &=\frac{m+n-1}{m+n-2}\vd (Xu)_{m}\\
\hd u &=\frac{-1}{m+n-2}\vd (Xu)_{m}\\
Xu &=(Xu)_{m}=X_{-}u_{m+1}.\\
\end{align*}
Using the commutator formula in Lemma \ref{lemma:co-lap} we may write
\begin{align*}
X\Delta u-\Delta Xu &=\frac{-2}{m+n-2}\vdiv\vd (Xu)_{m}+(n-1)Xu\\
&=2m(Xu)_{m}+(n-1)Xu\\
&=(2m+n-1)Xu.\\
\end{align*}
But since $\Delta Xu=\Delta (Xu)_{m}=m(m+n-2)Xu$ we obtain
\[X(\Delta u-(m(m+n)+n-1)u)=0.\]
The geodesic flow is Anosov and therefore transitive. This implies
\[\Delta u=(m(m+n)+n-1)u.\]
From this we easily derive that $u_k=0$ for all $k\neq m+1$. Thus $Xu=Xu_{m+1}=X_{-}u_{m+1}$ which implies that $X_{+}u_{m+1}=0$. It follows that $u_{m+1}$ is a conformal Killing tensor of order $m+1$. By hypothesis $u_{m+1}=0$ and hence $u=0$. This implies that $a$ has degree $m-1$ and $f=Xa$. Rewriting this in terms of symmetric tensors via the map $\lambda$ in Section \ref{sec:vlap+sh}, we see that $f = da$ is a potential tensor.
\end{proof}


\section{Manifolds with boundary} \label{sec_boundary}

As mentioned in the introduction, the methods in this paper also apply to manifolds with boundary. The main changes when going from closed manifolds to the boundary case include the assumption that test functions need to vanish on the boundary whenever appropriate, and that in the boundary case there are no nontrivial conformal Killing tensors vanishing on the boundary \cite{DS}. We now indicate what kinds of results can be achieved in the boundary case.

In this section we will assume that $(M,g)$ is a compact oriented Riemannian manifold with smooth boundary, and $n = \dim M\geq 2$. The results in Section \ref{sec_spherebundle} are valid in the boundary case, and the Pestov identity 
\begin{equation*}
\norm{\vd X u}^2 = \norm{X\vd u}^2-(R\,\vd u,\vd u) + (n-1)\norm{Xu}^2
\end{equation*}
now holds for any $u \in C^{\infty}(SM)$ with $u|_{\partial(SM)} = 0$. The spherical harmonics expansions $u = \sum_{m=0}^{\infty} u_m$ and the decomposition $X = X_+ + X_-$ in Section \ref{sec:vlap+sh} remain unchanged. One has the adjoint identity 
$$
(X_+ u, w) = -(u,X_- w)
$$
when $u \in \Omega_m$, $w \in \Omega_{m+1}$, and one of $u$, $w$ vanishes on $\partial(SM)$. Notice that in general one only needs boundary conditions in the $x$ variable, since one can integrate by parts freely in the $v$ variable (the fibres are compact with no boundary).

\begin{Lemma} \label{lemma_xplusminus_boundary}
Let $(M,g)$ be a compact manifold with smooth boundary.
\begin{enumerate}
\item[(a)]
If $X_+ u_m = 0$ where $u_m \in \Omega_m$, $u_m|_{\partial (SM)} = 0$, $m \geq 1$, then $u_m = 0$.
\item[(b)]
If $f_m \in \Omega_m$, $m \geq 0$, there exists $u_{m+1} \in \Omega_{m+1}$ with $X_- u_{m+1} = f_m$.
\item[(c)]
Any $f_m \in \Omega_m$, $m \geq 0$, has the orthogonal decomposition 
$$
f_m  = q_m + X_+ v_{m-1}, \qquad X_- q_m = 0, \ \ v_{m-1}|_{\partial(SM)} = 0.
$$
\end{enumerate}
\end{Lemma}
\begin{proof}
By the results of \cite{DS}, the problem  
$$
X_- X_+ u = f \text{ in } SM, \qquad u|_{\partial(SM)} = h
$$
is an elliptic boundary value problem on sections of trace-free symmetric $m$-tensors (hence on $\Omega_m$), with trivial kernel when $m \geq 1$. This shows (a) and also (b) since $X_- X_+: \{ u \in \Omega_m \,;\, u|_{\partial(SM)} = 0 \} \to \Omega_m$ is surjective for $m \geq 1$ (if $m=0$, the equation $X_- X_+ w_0 = f_0$ reduces to $d^* d w_0 = f_0$ which can always be solved). For (c), if $f_m \in \Omega_m$ it is enough to solve $X_- X_+ v_{m-1} = X_- f_m$ with $v_{m-1}|_{\partial(SM)}=0$ and to set $q_m = f_m - X_+ v_{m-1}$.
\end{proof}

As in Section \ref{sec_alphacontrolled}, the manifold with boundary $(M,g)$ is said to be $\alpha$-controlled if 
$$
\norm{X Z}^2 - (R Z, Z) \geq \alpha \norm{X Z}^2
$$
for all $Z \in \mathcal{Z}$ with $Z|_{\partial(SM)} = 0$. Lemmas \ref{lemma_xvu_lowerbound} and \ref{lemma_nablah_nablav_innerproduct} remain true, and Lemma \ref{lemma_alphacontrolled_properties} takes the following form.

\begin{Lemma} \label{lemma_boundary_alphacontrolled}
Let $(M,g)$ be a compact nontrapping manifold with strictly convex boundary (i.e.\ the second fundamental form of $\partial M \subset M$ is positive definite). Then 
\begin{itemize}
\item 
$(M,g)$ is $0$-controlled if it has no conjugate points,
\item 
$(M,g)$ is $\alpha$-controlled for some $\alpha > 0$ if it is simple (and moreover one has $\norm{X Z}^2 - (R Z, Z) \geq \alpha (\norm{X Z}^2 + \norm{Z}^2)$ for $Z \in \mathcal{Z}$ with $Z|_{\partial(SM)} = 0$),
\item 
$(M,g)$ is $1$-controlled iff it has nonpositive sectional curvature.
\end{itemize}
\end{Lemma}
\begin{proof}
If $(M,g)$ is as in the statement, one has the Santal\'o formula \cite[Lemma 3.3.2]{Sh2}
$$
\int_{SM} f = \int_{\partial_+(SM)} \int_0^{\tau(x,v)} f(\phi_t(x,v)) \abs{\langle v, \nu \rangle} \,dt \,d(\partial(SM)), \qquad f \in C^{\infty}(SM).
$$
Given  $Z \in \mathcal{Z}$ with $Z|_{\partial(SM)} = 0$, we define 
$$
Y_{x,v}(t) = Z(\phi_t(x,v)), \qquad (x,v) \in \partial_+(SM), \ t \in [0,\tau(x,v)].
$$
Then $Y_{x,v}$ is a vector field along the geodesic $\gamma_{x,v}: [0,\tau(x,v)] \to M$ starting from $(x,v)$, it is orthogonal to $\dot{\gamma}$ and vanishes at the endpoints, and we have 
$$
D_t Y_{x,v}(t) = XZ(\phi_t(x,v)).
$$
It follows from the Santal\'o formula that 
$$
\norm{X Z}^2 - (R Z, Z) = \int_{\partial_+(SM)} I_{\gamma_{x,v}}(Y_{x,v},Y_{x,v}) \abs{\langle v, \nu \rangle} \,d(\partial(SM))
$$
where $I_{\gamma}$ is the index form on a geodesic $\gamma: [0,T] \to M$, 
$$
I_{\gamma}(V,W) = \int_0^T \Big[ \langle D_t V, D_t W \rangle - \langle R(V,\dot{\gamma})\dot{\gamma}, W \rangle \Big] \,dt
$$
for normal vector fields $V$, $W$ along $\gamma$ that vanish at the endpoints.

Now if $(M,g)$ has no conjugate points, the index form is positive definite \cite[Lemma 4.3.1]{Jost} and thus $(M,g)$ is $0$-controlled by the above discussion. Similarly, if $(M,g)$ is simple, it was observed in \cite[proof of Theorem 7.1]{DKSaU} that there is $\eps > 0$ such that 
$$
I_{\gamma_{x,v}}(Y,Y) \geq \eps \int_0^{\tau(x,v)} \abs{Y}^2 \,dt \quad \text{uniformly over }(x,v) \in \partial_+(SM),
$$
for all normal vector fields $Y$ along $\gamma_{x,v}$ vanishing at the endpoints. Thus 
$$
\norm{X Z}^2 - (R Z, Z) \geq \eps \int_{\partial_+(SM)} \int_0^{\tau(x,v)} \abs{Z(\phi_t(x,v))}^2 \abs{\langle v, \nu \rangle} \,dt \,d(\partial(SM)) = \eps \norm{Z}^2
$$
for $Z \in \mathcal{Z}$ with $Z|_{\partial(SM)} = 0$. If now $C > 0$ is such that $(RZ,Z) \leq C \norm{Z}^2$ for $Z \in \mathcal{Z}$, and if $0 < \delta < 1$, then 
$$
\norm{X Z}^2 - (R Z, Z) \geq (1-\delta) \eps \norm{Z}^2 + \delta \norm{X Z}^2 - C \delta \norm{Z}^2
$$
which is $\geq \alpha (\norm{X Z}^2 + \norm{Z}^2)$ for some $\alpha > 0$ by choosing $\delta$ small enough. Finally, nonpositive sectional curvature is equivalent with having $(RZ,Z) \leq 0$ for all $Z \in \mathcal{Z}$ with $Z|_{\partial(SM)} = 0$, which is equivalent with $1$-controlled.
\end{proof}

\begin{Remark}{\rm The same proof as above shows that if there are no $\beta$-conjugate points ($\beta>0$), then $(M,g)$ is $\frac{\beta-1}{\beta}$-controlled.
The key ingredient is the positivity of the index form with parameter $\beta$ which was established for  example in \cite[p.46]{Pestov} or \cite[Lemma 5.3]{Dairbekov_nonconvex}.}
\label{rem:betac}
\end{Remark}

If $(M,g)$ is compact with boundary, by Lemma \ref{lemma_xplusminus_boundary} we can define the Beurling transform for any $m \geq 0$ by 
$$
B: \Omega_m \to \Omega_{m+2}, \ \ f_m \mapsto f_{m+2},
$$
where $f_{m+2} \in \Omega_{m+2}$ is the unique solution of the equation 
$$
X_- f_{m+2} = - X_+ f_m
$$
that is orthogonal to $\mathrm{Ker}(X_-)$ (equivalently, the unique solution with minimal $L^2$ norm, or the unique solution of the form $X_+ v_{m+1}$ where $v_{m+1}|_{\partial(SM)} = 0$). Lemma \ref{lemma_beurling_n_neq_three} remains true if one considers $u \in \Omega_m$ with $u|_{\partial(SM)} = 0$. The arguments in Section \ref{section_beurling} now show that in nonpositive curvature, the Beurling transform is essentially a contraction and formal invariant distributions starting at any $f$ exist.

\begin{Theorem} \label{theorem_boundary_beurling}
Let $(M,g)$ be a compact manifold with boundary, and assume that the sectional curvatures are nonpositive. 
The Beurling transform satisfies
$$
\norm{Bf}_{L^2} \leq C_n(m) \norm{f}_{L^2}, \quad f \in \Omega_m,
$$
where 
\begin{align*}
C_2(m) &= \left\{ \begin{array}{cl} \sqrt{2}, & m = 0, \\ 1, & m \geq 1, \end{array} \right. \\
C_3(m) &= \left[ 1 + \frac{1}{(m+2)^2(2m+1)} \right]^{1/2}, \\
C_n(m) &\leq 1 \ \text{ for $n \geq 4$.}
\end{align*}
If $m_0 \geq 0$ and if $f \in \Omega_{m_0}$ satisfies $X_- f = 0$, then there is a solution of 
$$
Xw = 0 \text{ in } SM, \qquad w_{m_0} = f,
$$
given by $w = \sum_{k=0}^{\infty} B^k f$. One has $w \in L^2_x H^{-1/2-\eps}_{v\phantom{\theta}}(SM)$ for any $\eps > 0$, and the Fourier coefficients of $w$ satisfy 
$$
\norm{w_{m_0+2k}}_{L^2} \leq A_n(m_0) \norm{f}_{L^2}, \quad k \geq 0,
$$
where $A_n(m_0) = \prod_{j=0}^{\infty} C_n(m_0+2j)$ is a finite constant satisfying 
\begin{align*}
A_2(m_0) &= \left\{ \begin{array}{cl} \sqrt{2}, & m_0 = 0, \\ 1, & m_0 \geq 1, \end{array} \right. \\
A_3(m_0) &\leq 1.13, \\
A_n(m_0) &\leq 1 \ \text{ for $n \geq 4$.}
\end{align*}

\end{Theorem}

We have seen in Lemma \ref{lemma_boundary_alphacontrolled} that simple manifolds are always $\alpha$-controlled for some positive $\alpha$. For these manifolds, the methods in Sections \ref{sec_surjectivity_izero} and \ref{sec_surjectivity_im} yield the following subelliptic estimates and existence results for invariant distributions.

\begin{Theorem} \label{theorem_boundary_subelliptic}
Let $(M,g)$ be a simple manifold. Then 
$$
\norm{u}_{H^1} \lesssim \norm{Pu}_{L^2}, \qquad u \in C^{\infty}(SM), \ u|_{\partial(SM)} = 0.
$$
If $f \in L^2(M)$, there exists $w \in L^2_x H^{-1}_v(SM)$ satisfying 
$$
Xw = 0 \text{ in } SM, \qquad w_0 = f.
$$
\end{Theorem}

\begin{Theorem} \label{theorem_boundary_subelliptic2}
Let $(M,g)$ be a simple manifold, let $m \geq 1$, and let $(M,g)$ be $\alpha$-controlled for $\alpha > \alpha_{m,n}$ where 
$$
\alpha_{m,n} = \frac{(m-1)(m+n-2)}{m(m+n-1)}.
$$
Then we have 
$$
\norm{u}_{H^1} \lesssim \norm{Q_m u}_{L^2}, \qquad u \in T_{\geq m} C^{\infty}(SM), \ u|_{\partial(SM)} = 0.
$$
If $a_m \in \Omega_m$ with $X_- a_m = 0$, there is $w \in H^{-1}(SM)$ with $Xw = 0$ and $w_m = a_m$.
\end{Theorem}

\begin{Remark}{\rm 
In the case of simple manifolds, one actually expects to find invariant distributions that are $C^{\infty}$ functions. (Invariant distributions are solutions of $Xw = 0$ in $SM$, and on simple manifolds with boundary there are many solutions in $C^{\infty}(SM)$. In contrast, on closed manifolds with ergodic geodesic flow any solution $w \in L^2(SM)$ is constant and thus nontrivial solutions must be distributions.) One has the following results in the direction of Theorems \ref{theorem_boundary_beurling}--\ref{theorem_boundary_subelliptic2}:
\begin{itemize}
\item 
If $(M,g)$ is simple and $f_0 \in C^{\infty}(M)$, there is $w \in C^{\infty}(SM)$ with $Xw = 0$ in $SM$ and $w_0 = f_0$ \cite{PU}.
\item 
If $(M,g)$ is simple and $f_1 \in \Omega_1$ with $X_- f_1 = 0$, there is $w \in T_{\geq 1} C^{\infty}(SM)$ with $Xw = 0$ in $SM$ and $w_1 = f_1$ \cite{DU}.
\item 
If $(M,g)$ is a simple surface and if $f_m \in \Omega_m$ where $X_- f_m = 0$ and $m \geq 1$, there is $w \in T_{\geq m} C^{\infty}(SM)$ with $Xw = 0$ in $SM$ and $w_m = f_m$ \cite{PSU_range}.
\end{itemize}
These results rely on the ellipticity of the normal operator $I_m^* I_m$ and a solenoidal extension argument. If $(M,g)$ is simple, the ellipticity of $I_m^* I_m$ acting on solenoidal tensor fields is known for any $m$ \cite{SSU}, but it is not clear how to perform the solenoidal extension as in \cite{DU}. Thus, at present, we are not able to produce $C^{\infty}$ invariant distributions in the setting of Theorem \ref{theorem_boundary_subelliptic2} if $\dim M \geq 3$ (although a weaker result similar to \cite[Theorem 4.2]{PU_IMRN} seems possible).}
\end{Remark}

Finally, modifying the arguments in Section \ref{section_injectivity} appropriately yields the following solenoidal injectivity result.

\begin{Theorem} \label{theorem_boundary_im_injectivity}
Let $(M,g)$ be simple, let $m \geq 2$, and let $(M,g)$ be $\alpha$-controlled for $\alpha \geq \alpha_{m,n}$ where 
$$
\alpha_{m,n} = \frac{(m-1)(m+n-2)}{m(m+n-1)}.
$$
Then $I_m$ is $s$-injective.
\end{Theorem}

\begin{Remark}
{\rm The same argument proving Theorem \ref{theorem_boundary_im_injectivity} also gives the following result: let $(M,g)$ be a compact manifold with non-empty boundary, nonpositive sectional curvature and with property that any smooth $q$ with $Xq=0$ and $q|_{\partial(SM)} = 0$ must vanish. If $u \in C^{\infty}(SM)$ solves $Xu = f$ in $SM$ with $u|_{\partial(SM)} = 0$ where $f \in C^{\infty}(SM)$ has degree $m$, then $u$ has degree $m-1$.}
\end{Remark}

Theorem \ref{theoremD} in the Introduction follows directly from Theorem \ref{theorem_boundary_im_injectivity} and Remark \ref{rem:betac}.

\appendix

\section{Proofs of the commutator identities} \label{sec_localcoordinates}

In this appendix we give the proofs of the commutator identities in Section \ref{sec_spherebundle}. This is done via local coordinate computations, and we also give coordinate expressions for the relevant operators which have been defined invariantly in Section \ref{sec_spherebundle}. The arguments are not new and they arise in the calculus of semibasic tensor fields as in \cite{Sh}. The main points here are that the basic setting is the unit sphere bundle $SM$ instead of $TM$, and that all computations can be done on the level of vector fields instead of (higher order) semibasic tensor fields.

\vspace{10pt}

\noindent {\bf Vector fields on $SM$.} If $x$ is a system of local coordinates in $M$, let $(x,y)$ be associated coordinates in $TM$ where tangent vectors are written as $y^j \partial_{x_j}$. One has corresponding coordinates $(x,y,X,Y)$ in $T(TM)$ where vectors of $T(TM)$ are written as $X^j \partial_{x_j} + Y^j \partial_{y_j}$. It is convenient to introduce the vector fields 
$$
\delta_{x_j} = \partial_{x_j} - \Gamma_{jk}^l y^k \partial_{y_l}
$$
where $\Gamma_{jk}^l$ are the Christoffel symbols of $(M,g)$. The Sasaki metric on $TM$ is expressed in local coordinates as 
$$
\langle X^j \delta_{x_j} + Y^k \partial_{y_k}, \tilde{X}^j \delta_{x_j} + \tilde{Y}^k \partial_{y_k} \rangle = g_{jk} X^j \tilde{X}^k + g_{jk} Y^j \tilde{Y}^k.
$$
The horizontal and vertical subbundles are spanned by $\{ \delta_{x_j} \}_{j=1}^n$ and $\{ \partial_{y_k} \}_{k=1}^n$, respectively. It will be very convenient to identify horizontal and vertical vector fields on $TM$ with vector fields on $M$ via the maps $X^j \delta_{x_j} \mapsto X^j \partial_{x_j}$ and $Y^k \partial_{y_k} \mapsto Y^k \partial_{x_k}$ (see for example \cite{Pa} for more details), and we will use this identification freely below. We will also raise and lower indices with respect to the metric $g_{jk}$.

The hypersurface $SM$ in $TM$ is given by $SM = f^{-1}(1)$ where $f: TM \to \mR$ is the function $f(x,y) = g_{jk}(x) y^j y^k$. A computation gives 
$$
df(X^j \delta_{x_j} + Y^k \partial_{y_k}) = 2 y_k Y^k.
$$
Then $T(SM)$ is the subset of $T(TM)$ given by 
$$
T(SM) = \{ X^j \delta_{x_j} + Y^k \partial_{y_k} \in T(TM) \,;\, (x,y) \in SM, \ y_k Y^k = 0 \}.
$$
To be precise, we identify vector fields $V$ on $SM$ with the corresponding fields $i_* V$ on $TM$, where $i: SM \to TM$ is the natural inclusion. Equip $SM$ with the restriction of the Sasaki metric from $TM$. The identity $d_{SM} i^* U = i^* d_{TM} U$ for functions $U$ on $TM$ implies that the gradient on $SM$ is given by 
$$
\langle \nabla_{SM} u, V \rangle = (\delta_{x_j} \tilde{u}) X^j + (\partial_{y_k} \tilde{u}) Y^k, \qquad u \in C^{\infty}(SM),
$$
where $\tilde{u} \in C^{\infty}(TM)$ is any function with $\tilde{u}|_{SM} = u$ and where the vector field $V$ on $SM$ is expressed as above in the form $V = X^j \delta_{x_j} + Y^k \partial_{y_k}$.

We define vector fields on $SM$ that act on $u \in C^{\infty}(SM)$ by 
\begin{align*}
\delta_j u &= \delta_{x_j} (u \circ p)|_{SM}, \\
\partial_k u &= \partial_{y_k} (u \circ p)|_{SM}
\end{align*}
where $p: TM \setminus \{0\} \to SM$ is the projection $p(x,y) = (x,y/\abs{y}_{g(x)})$. We see that the decomposition $\nabla_{SM} u = (Xu)X + \hd u + \vd u$ has the following form in local coordinates:
\begin{align*}
Xu &= v^j \delta_j u, \\
\hd u &= (\delta^j u - (v^k \delta_k u) v^j) \partial_{x_j}, \\
\vd u &= (\partial^k u) \partial_{x_k}.
\end{align*}

\vspace{10pt}

\noindent {\bf Commutator formulas.} Direct computations in local coordinates give the following formulas for vector fields on $TM$:
\begin{gather*}
[\delta_{x_j}, \delta_{x_k}] = -R_{jkl}^{\phantom{jkl}m} y^l \partial_{y_m}, \qquad [\delta_{x_j}, \partial_{y_k}] = \Gamma_{jk}^l \partial_{y_l}, \qquad [\partial_{y_j}, \partial_{y_k}] = 0.
\end{gather*}
We wish to consider corresponding formulas for the vector fields $\delta_j$ and $\partial_j$ on $SM$. If $u \in C^{\infty}(SM)$, write $\tilde{u}(x,y) = (u \circ p)(x,y) = u(x,y/\abs{y}_g)$. Homogeneity implies that 
$$
(\delta_j u)\etilde(x,y) = \delta_{x_j} \tilde{u}(x,y), \qquad (\partial_k u)\etilde(x,y) = \abs{y}_g \partial_{y_k} \tilde{u}(x,y).
$$
Since also $\delta_{x_j}(\abs{y}_g) = 0$ and $\partial_{y_j}(\abs{y}_g) = y_j/\abs{y}_g$, we obtain 
\begin{gather*}
[\delta_j, \delta_k] = - R_{jklm} v^l \partial^m, \qquad [\delta_j, \partial_k] = \Gamma_{jk}^l \partial_l , \qquad [\partial_j, \partial_k] = v_j \partial_k - v_k \partial_j, \\
[\partial_j, v^k] = \delta_j^{\phantom{j}k} - v_j v^k, \qquad [\delta_j, v^l] = -\Gamma_{jk}^l v^k.
\end{gather*}
We also note that 
$$
v^j \partial_j = 0.
$$
Using the identity $\partial_{x_j} g^{ab} + g^{am} \Gamma_{jm}^b + g^{bm} \Gamma_{jm}^a = 0$ we also obtain 
$$
[\delta_j, \delta^l] = -g^{lk} R_{jkpq} v^p \partial^q + (\partial_{x_j} g^{lk}) \delta_k, \qquad [\delta_j, \partial^l] = -\Gamma_{jk}^l \partial^k.
$$

We can now compute $[X, \vd]$. Note that 
$$
\vd Xu = \partial^l(v^j \delta_j u) \partial_{x_l} = \hd u + v^j \partial^l \delta_j u \partial_{x_l}
$$
and (by the formula \eqref{xz_local_coordinates} below) 
$$
X \vd u = (v^j \delta_j \partial^l u + \Gamma_{jk}^l v^j \partial^k u) \partial_{x_l}.
$$
Thus 
$$
[X, \vd]u = - \hd u + v^j ( [\delta_j, \partial^l] u + \Gamma_{jk}^l \partial^k u) = -\hd u.
$$

Moving on to $[X, \hd]$, we observe that 
$$
\hd X u = (\delta^l(Xu) - (X^2 u) v^l) \partial_{x_l}
$$
and (again by \eqref{xz_local_coordinates}) 
$$
X \hd u = (X (\delta^l u - (Xu) v^l) + \Gamma_{jk}^l v^j (\delta^k u - (Xu) v^k)) \partial_{x_l}.
$$
In the second term we have $-(Xu)(Xv^l) - \Gamma_{jk}^l v^j (Xu) v^k = 0$, and when taking the commutator the terms containing $X^2 u$ cancel. It follows that 
\begin{align*}
[X, \hd] u &= (v^j \delta_j \delta^l u - \delta^l(v^j \delta_j u) + \Gamma_{jk}^l v^j \delta^k u) \partial_{x_l} \\
 &= (v^j [\delta_j, \delta^l] u - g^{lr} [\delta_r, v^j] \delta_j u + \Gamma_{jk}^l v^j \delta^k u) \partial_{x_l} \\
 &= (- g^{lk} R_{jkpq} v^j v^p \partial^q u + [ (\partial_j g^{lk}) v^j \delta_k u + g^{lr} \Gamma_{rm}^j v^m \delta_j u + g^{km} \Gamma_{jk}^l v^j \delta_m u ] ) \partial_{x_l}.
\end{align*}
The part in brackets is zero, which yields $[X, \hd] u = R_{abc}^{\phantom{abc}l} (\partial^a u) v^b v^c \partial_{x_l} = R \vd u$.

\vspace{10pt}

\noindent {\bf Adjoints.} To prove the last basic commutator formula, it is useful to have local coordinate expressions for the adjoints of $X, \hd, \vd$ on the space $\mathcal Z$. The first step is to compute the adjoints of the local vector fields $\delta_j$ and $\partial_j$: if $u, w \in C^{\infty}(SM)$ and $w$ vanishes when $x$ is outside a coordinate patch (and additionally $w$ vanishes on $\partial(SM)$ if $M$ has a boundary), we claim that in the $L^2(SM)$ inner product 
\begin{equation} \label{deltaj_partialj_adjoint}
(\delta_j u, w) = -(u, (\delta_j + \Gamma_j) w), \qquad (\partial_j u, w) = -(u, (\partial_j -(n-1)v_j) w).
\end{equation}
Here $\Gamma_j = \Gamma_{jk}^k$.

Assuming these identities, one can check that the adjoint of $X$ on $C^{\infty}(SM)$ is $-X$. Moreover, if $Z \in \mathcal Z$ is written as $Z(x,v) = Z^j(x,v) \partial_{x_j}$, the vector field $XZ$ is the covariant derivative (with respect to the Levi-Civita connection in $(M,g)$) 
\begin{equation} \label{xz_local_coordinates}
XZ(x,v) = D_t(Z(\phi_t(x,v)))|_{t=0} = (XZ^j) \partial_{x_j} + \Gamma_{jk}^l v^j Z^k \partial_{x_l}.
\end{equation}
Then the adjoint of $X$ on $\mathcal Z$ is also $-X$. The adjoints of $\hd$ and $\vd$ are given in local coordinates by $-\hdiv$ and $-\vdiv$, where 
$$
\hdiv Z = (\delta_j + \Gamma_j) Z^j, \qquad \vdiv Z = \partial_j Z^j.
$$
Given these expressions, we get the final commutator formula: 
\begin{align*}
\hdiv \vd u - \vdiv \hd u &= (\delta_k + \Gamma_k)(\partial^k u) - \partial_k(\delta^k u - (Xu) v^k) \\
 &= [\delta_k, \partial^k] u + \Gamma_k \partial^k u + (Xu)(\partial_k v^k) \\
 &= (n-1) Xu.
\end{align*}

It remains to check \eqref{deltaj_partialj_adjoint}. To do this it is enough to prove that 
$$
\partial_{x_j} \left( \int_{S_x M} u \,dS_x \right) = \int_{S_x M} \delta_j u \,dS_x, \qquad \int_{S_x M} \partial_j u \,dS_x = (n-1) \int_{S_x M} u v_j \,dS_x.
$$
(The proof of \eqref{deltaj_partialj_adjoint} also uses the identity $\abs{g}^{-1/2} \partial_{x_j} (\abs{g}^{1/2}) = \Gamma_j$.) To show the first formula, let $u \in C^{\infty}(SM)$ vanish when $x$ is outside a coordinate patch, and define 
$$
f(x) = \int_{S_x M} u(x,v) \,dS_x(v).
$$
Write $\tilde{u}(x,y) = u(x,y/\abs{y}_g)$, and choose $\varphi \in C^{\infty}_c((0,\infty))$ so that $\int_0^{\infty} \varphi(r) r^{n-1} \,dr = 1$. Write $g(x)$ for the matrix of $g$ in the $x$ coordinates. We write $\omega$ for points in $S^{n-1}$, and note that $\omega \mapsto g(x)^{-1/2} \omega$ is an isometry from $S^{n-1}$ (with the metric induced by the Euclidean metric $e$ in $\mR^n$) onto $S_x M$ (with the metric induced by Sasaki metric on $T_x M$, having volume form $dT_x = \abs{g(x)}^{1/2} \,dx$). Therefore 
\begin{align*}
f(x) &= \int_0^{\infty} \varphi(r) r^{n-1} \int_{S_x M} u(x,v) \,dS_x(v) \,dr \\
 &= \int_0^{\infty} \int_{S^{n-1}} \varphi(r) r^{n-1} \tilde{u}(x,g(x)^{-1/2} r\omega) \,d\omega \,dr \\
 &= \int_{\mR^n} \varphi(\abs{\eta}_e) \tilde{u}(x,g(x)^{-1/2} \eta) \,d\eta \\
 &= \int_{\mR^n} \varphi(\abs{y}_g) \tilde{u}(x,y) \abs{g(x)}^{1/2} \,dy.
\end{align*}
Since $\abs{g}^{-1/2} \partial_{x_j} (\abs{g}^{1/2}) = \Gamma_j = \partial_{y_l}(\Gamma_{jk}^l y^k)$, we have 
\begin{align*}
\partial_{x_j} f(x) &= \int_{\mR^n} (\partial_{x_j} + \partial_{y_l}(\Gamma_{jk}^l y^k)) \left[ \varphi(\abs{y}_g) \tilde{u}(x,y) \right] \abs{g(x)}^{1/2} \,dy \\
 &= \int_{\mR^n} \delta_{x_j} \left[ \varphi(\abs{y}_g) \tilde{u}(x,y) \right] \abs{g(x)}^{1/2} \,dy.
\end{align*}
Now $\delta_{x_j}(\abs{y}_g) = 0$, and it follows by undoing the changes of variables above that 
$$
\partial_{x_j} f(x) = \int_{S_x} \delta_j u(x,v) \,dS_x(v)
$$
as required. The second formula follows from a similar computation as above: now we define 
$$
f(x) = \int_{S_x M} \partial_j u(x,v) \,dS_x(v)
$$
and compute 
\begin{align*}
f(x) &= \int_0^{\infty} \varphi(r) r^{n-1} \int_{S_x M} \partial_j u(x,v) \,dS_x(v) \,dr \\
 &= \int_0^{\infty} \int_{S^{n-1}} r \varphi(r) r^{n-1} \partial_{y_j} \tilde{u}(x,g(x)^{-1/2} r\omega) \,d\omega \,dr \\
 &= \int_{\mR^n} \abs{y}_g \varphi(\abs{y}_g) \partial_{y_j} \tilde{u}(x,y) \abs{g(x)}^{1/2} \,dy.
\end{align*}
Write $h(y) = \abs{y}_g \varphi(\abs{y}_g)$. Then 
\begin{align*}
f(x) &= -\int_{\mR^n} \partial_{y_j} h(y) \tilde{u}(x,y) \abs{g(x)}^{1/2} \,dy \\
 &= -\int_{S_x M} \left[ \int_0^{\infty} \partial_{y_j} h(rv) r^{n-1} \,dr \right] u(x,v) \,dS_x(v).
\end{align*}
The expression in brackets is $v_j \int_0^{\infty} (\varphi(r) + r \varphi'(r))r^{n-1} \,dr = -(n-1) v_j$, and the result follows.

\section{The two-dimensional case} \label{sec_2d_case}

In this section we reconsider the arguments in this paper in the special case of two-dimensional manifolds. This discussion allows to connect the present treatment with earlier work in two dimensions, in particular \cite{PSU1}, \cite{PSU3}.

\vspace{10pt}

\noindent {\bf Vector fields.} Let $(M,g)$ be a compact oriented Riemann surface with no boundary (the boundary case is analogous, if we additionally assume that test functions vanish on $\partial(SM)$ if appropriate). We have $n = \dim(M) = 2$. For any vector $v \in S_x M$, there is a unique vector $iv \in S_x M$ such that $\{ v, iv \}$ is a positive orthonormal basis of $T_x M$. Let $X$ be the geodesic vector field on $SM$ as before. We can define vector fields $V$ and $X_{\perp}$ on $SM$, acting on $u \in C^{\infty}(SM)$ by 
\begin{align*}
X_{\perp} u(x,v) &:= -\langle \hd u(x,v), iv \rangle, \\
Vu(x,v) &:= \langle \vd u(x,v), iv \rangle.
\end{align*}
Note that since $\langle \vd u, v \rangle = \langle \hd u, v \rangle = 0$, we have 
\begin{align*}
\hd u &= -(X_{\perp} u) iv, \\
\vd u &= (Vu) iv.
\end{align*}
Note also that any $Z \in \mathcal Z$ is of the form $Z(x,v) = z(x,v) iv$ for some $z \in C^{\infty}(SM)$. If $\gamma(t)$ is a unit speed geodesic we have 
\begin{align*}
\langle D_t [i \dot{\gamma}(t)], \dot{\gamma}(t) \rangle &= \partial_t \langle i \dot{\gamma}(t), \dot{\gamma}(t) \rangle = 0, \\
\langle D_t [i \dot{\gamma}(t)], i \dot{\gamma}(t) \rangle &= \frac{1}{2} \partial_t \langle i \dot{\gamma}(t), i \dot{\gamma}(t) \rangle = 0.
\end{align*}
Thus $iv$ is parallel along geodesics, and for $Z = z(x,v) iv$ we have $XZ = (Xz) iv$. 

The Guillemin-Kazhdan operators \cite{GK} are defined as the vector fields 
$$
\eta_{\pm} := \frac{1}{2}(X \pm i X_{\perp}).
$$
All these vector fields have simple expressions in isothermal coordinates. Since $(M,g)$ is two-dimensional, near any point there are positively oriented isothermal coordinates $(x_{1},x_{2})$ so that the metric can be written as $ds^2=e^{2\lambda}(dx_{1}^2+dx_{2}^2)$ where $\lambda$ is a smooth real-valued function of $x=(x_{1},x_{2})$. This gives coordinates $(x_{1},x_{2},\theta)$ on $SM$ where $\theta$ is the angle between a unit vector $v$ and $\partial/\partial x_{1}$. In these coordinates we have 
\begin{gather*}
V = \frac{\partial}{\partial \theta}, \qquad X = e^{-\lambda} \left[ \cos\theta\frac{\partial}{\partial x_{1}}+ \sin\theta\frac{\partial}{\partial x_{2}}+ \left(-\frac{\partial \lambda}{\partial x_{1}}\sin\theta+\frac{\partial\lambda}{\partial x_{2}}\cos\theta\right)\frac{\partial}{\partial \theta} \right], \\
X_{\perp} =-e^{-\lambda}\left[-\sin\theta\frac{\partial}{\partial x_{1}}+ \cos\theta\frac{\partial}{\partial x_{2}}- \left(\frac{\partial \lambda}{\partial x_{1}}\cos\theta+\frac{\partial \lambda}{\partial x_{2}}\sin\theta\right)\frac{\partial}{\partial \theta}\right], \\
\eta_+ = e^{-\lambda} e^{i\theta} \left[ \frac{\partial}{\partial z} + i \frac{\partial \lambda}{\partial z} \frac{\partial}{\partial \theta} \right], \qquad \eta_- = e^{-\lambda} e^{-i\theta} \left[ \frac{\partial}{\partial \bar{z}} - i \frac{\partial \lambda}{\partial \bar{z}} \frac{\partial}{\partial \theta} \right]
\end{gather*}
where $\partial/\partial z = \frac{1}{2}(\partial_{x_1} - i \partial_{x_2})$ and $\partial/\partial \bar{z} = \frac{1}{2}(\partial_{x_1} + i \partial_{x_2})$.

\vspace{10pt}

\noindent {\bf Commutator formulas.} The above discussion shows that the commutator formula $[X, \vd] = -\hd$ reduces to 
$$
[X,V] = X_{\perp}.
$$
Since $(R \vd u)(x,v) = K(x)(Vu) iv$ where $K$ is the Gaussian curvature, the commutator formula $[X, \hd] = R\vd$ becomes 
$$
[X, X_{\perp}] = -KV.
$$
For the last commutator formula we need to compute $\hdiv$ and $\vdiv$. First observe that a local coordinate computation gives for $w \in C^{\infty}(SM)$ that 
$$
\int_{SM} Vw = \int_M \int_{S_x M} (iv)^j \partial_j w = 0.
$$
Thus the adjoint of $V$ is $-V$, and the first commutator formula implies that the adjoint of $X_{\perp}$ is $-X_{\perp}$. Consequently 
\begin{align*}
\hdiv (z(x,v) iv) &= -X_{\perp} z, \\
\vdiv (z(x,v) iv) &= Vz.
\end{align*}
The commutator formula $\hdiv\,\vd-\vdiv\,\hd=(n-1)X$ thus reduces to 
$$
[V,X_{\perp}] = X.
$$

\vspace{10pt}

\noindent {\bf Spherical harmonics expansions.} It is easy to express the operators $X_{\pm}$ in terms of $\eta_{\pm}$. The vertical Laplacian on $SM$ is given by 
$$
-\vdiv \vd u = -\vdiv((Vu) iv) = -V^2 u.
$$
The operator $-iV$ on $L^2(SM)$ has eigenvalues $k \in \mZ$ with corresponding eigenspaces $E_k$. We write 
$$
L^2(SM) = \bigoplus_{k=-\infty}^{\infty} E_k, \qquad u = \sum_{k=-\infty}^{\infty} u_k.
$$
Locally in the $(x,\theta)$ coordinates, elements of $E_k$ are of the form $\tilde{w}(x) e^{ik\theta}$. Writing $\Lambda_k = C^{\infty}(SM) \cap E_k$, the spherical harmonics of degree $m$ are given by 
$$
\Omega_m = \Lambda_m \oplus \Lambda_{-m}, \qquad m \geq 0.
$$
If $m \geq 1$, the action of $X_{\pm}$ on $\Omega_m$ is given by 
$$
X_{\pm}(e_m+e_{-m}) = \eta_{\pm} e_m + \eta_{\mp} e_{-m}, \quad e_j \in \Lambda_j,
$$
and for $m=0$ we have $X_+|_{\Omega_0} = \eta_+ + \eta_-$, $X_-|_{\Omega_0} =0$. In the two-dimensional case it will be convenient to work with the $\Lambda_k$ spaces (the corresponding results in terms of the $\Omega_m$ spaces will follow easily).

\vspace{10pt}

\noindent {\bf Beurling transform and invariant distributions.} Recall that $w \in \mDp(SM)$ is called invariant if $Xw = 0$. If the geodesic flow is ergodic, these are genuinely distributions since any $w \in L^1(SM)$ that satisfies $Xw = 0$ must be constant. For Riemann surfaces, one can look at distributions with one-sided Fourier series; let us consider the case where $w_k = 0$ for $k < k_0$, for some integer $k_0 \geq 0$. For such a distribution, the equation $Xw = 0$ reduces to countably many equations for the Fourier coefficients (by parity it is enough to look at $w_{k_0+2j}$ for $j \geq 0$):
\begin{align*}
\eta_- w_{k_0} &= 0, \\
\eta_- w_{k_0+2} &= -\eta_+ w_{k_0}, \\
\eta_- w_{k_0+4} &= -\eta_+ w_{k_0+2}, \\
 &\ \,\vdots
\end{align*}

On a Riemann surface with genus $\geq 2$, the operator $\eta_+: \Lambda_{k-1} \to \Lambda_k$ is injective and its adjoint $\eta_-: \Lambda_k \to \Lambda_{k-1}$ is surjective for $k \geq 2$ by conformal invariance (there is a constant negative curvature metric in the conformal class, and these have no conformal Killing tensors). Also for $k \geq 2$ we have the $L^2$-orthogonal splitting 
$$
\Lambda_k = \mathrm{Ker}(\eta_-|_{\Lambda_k}) \oplus \eta_+ \Lambda_{k-1}.
$$
If $k \geq 0$, we define the \emph{Beurling transform} 
$$
B_+: \Lambda_k \to \Lambda_{k+2}, \ \ f_k \mapsto f_{k+2}
$$
where $f_{k+2}$ is the unique function in $\Lambda_{k+2}$ orthogonal to $\mathrm{Ker}(\eta_-|_{\Lambda_{k+2}})$ (equivalently, the $L^2$-minimal solution) satisfying $\eta_- f_{k+2} = -\eta_+ f_k$. Note that in $\re^2$, one thinks of $\eta_-$ as $\dbar$ and of $\eta_+$ as $\partial$, so $B_+$ is formally the operator $-\dbar^{-1} \partial$ which is the usual Beurling transform up to minus sign. Note also that $B_+$ is the first ladder operator from \cite{GK1}.

If $k \geq 0$ one has the analogous operator 
$$
B_-: \Lambda_{-k} \to \Lambda_{-k-2}, \ \ f_{-k} \mapsto f_{-k-2}
$$
where $f_{-k-2}$ is the $L^2$-minimal solution of $\eta_+ f_{-k-2} = -\eta_- f_{-k}$. The relation to the Beurling transform in Section \ref{section_beurling} is 
$$
B(f_k + f_{-k}) = B_+ f_k + B_- f_{-k}, \qquad f_j \in \Lambda_j.
$$

On a closed surface of genus $\geq 2$, one can always formally solve the countably many equations for $w_k$. If we take the minimal energy solution for each equation, we arrive at the formal invariant distributions. We restrict our attention to $B_+$ (the case of $B_-$ is analogous).

\begin{Definition}
Let $(M,g)$ be a closed oriented surface with genus $\geq 2$, let $k_0 \geq 0$, and let $f \in \Lambda_{k_0}$ satisfy $\eta_- f = 0$. The \emph{formal invariant distribution} starting at $f$ is the formal sum 
$$
w = \sum_{j=0}^{\infty} (B_+)^j f.
$$
\end{Definition}

As before, it is not clear if the sum converges in any reasonable sense. However, if the surface has nonpositive curvature it does converge nicely. This follows from the fact that the Beurling transform $B_+$ is a contraction on such surfaces, and to prove this we use the Guillemin-Kazhdan energy identity \cite{GK}:

\begin{Lemma} \label{lemma_gk_energy_identity}
Let $(M,g)$ be a closed Riemann surface. Then 
$$
\norm{\eta_- u}^2 = \norm{\eta_+ u}^2 - \frac{i}{2} (K V u, u), \qquad u \in C^{\infty}(SM).
$$
\end{Lemma}
\begin{proof}
In \cite{GK} one has the commutator formula 
$$
[\eta_+, \eta_-] = \frac{i}{2} K V.
$$
This implies that, for $u \in C^{\infty}(SM)$, 
$$
\norm{\eta_- u}^2 = \norm{\eta_+ u}^2 + ([\eta_-,\eta_+]u,u) = \norm{\eta_+ u}^2 - \frac{i}{2} (K V u, u). \qedhere
$$
\end{proof}

\begin{Lemma} \label{lemma_twodim_beurling}
Let $(M,g)$ be a closed surface, and assume that $K \leq 0$. Then for any $k \geq 0$  we have 
$$
\norm{\eta_- u}_{L^2} \leq \norm{\eta_+ u}_{L^2}, \quad u \in \Lambda_k,
$$
and 
$$
\norm{B_+ f}_{L^2} \leq \norm{f}_{L^2}, \quad f \in \Lambda_k.
$$
If $k_0 \geq 0$ and if $f \in \Lambda_{k_0}$ satisfies $\eta_- f = 0$, then the formal invariant distribution $w$ starting at $f$ is an element of $L^2_x H^{-1/2-\eps}_{\theta}$ for any $\eps > 0$. Moreover, the Fourier coefficients of $w$ satisfy 
$$
\norm{w_k}_{L^2} \leq \norm{f}_{L^2}, \quad k \geq 0.
$$
\end{Lemma}
\begin{proof}
Lemma \ref{lemma_gk_energy_identity} implies that for $u \in \Lambda_k$ with $k \geq 0$,
$$
\norm{\eta_- u}^2 = \norm{\eta_+ u}^2 - \frac{i}{2} (K V u, u) = \norm{\eta_+ u}^2 + \frac{k}{2} (Ku,u).
$$
Using that $K \leq 0$ and $k \geq 0$ we get $\norm{\eta_- u} \leq \norm{\eta_+ u}$. The rest of the claims follow as in Section \ref{section_beurling}.
\end{proof}

The next lemma considers $B$ instead of $B_+$ and shows that the constants in the first inequality are sharp for flat Riemann surfaces.

\begin{Lemma} \label{lemma_twodim_beurling2}
Let $(M,g)$ be a closed surface, and assume that $K \leq 0$. Then 
$$
\norm{X_- u}_{L^2} \leq \left\{ \begin{array}{ll} \norm{X_+ u}_{L^2}, & u \in \Omega_m \text{ with } m \geq 2, \\ \sqrt{2} \norm{X_+ u}_{L^2}, & u \in \Omega_1. \end{array} \right.
$$
If $K=0$ the constants are sharp. If additionally the genus is $\geq 2$ (so there are no conformal Killing tensors), then 
$$
\norm{Bf}_{L^2} \leq \left\{ \begin{array}{ll} \norm{f}_{L^2}, & f \in \Omega_m \text{ with } m \geq 1, \\ \sqrt{2} \norm{f}_{L^2}, & f \in \Omega_0. \end{array} \right.
$$
\end{Lemma}
\begin{proof}
If $u \in \Omega_m$ with $m \geq 2$, then $u = f_m + f_{-m}$ with $f_j \in \Lambda_j$. Lemma \ref{lemma_gk_energy_identity} yields 
\begin{align*}
\norm{X_- u}^2 &= \norm{\eta_- f_m}^2 + \norm{\eta_+ f_{-m}}^2 \\
 &= \norm{\eta_+ f_m}^2 + \norm{\eta_- f_{-m}}^2 + \frac{m}{2} ((K f_m, f_m) + (K f_{-m}, f_{-m})).
\end{align*}
Using that $K \leq 0$, this gives $\norm{X_- u}^2 \leq \norm{X_+ u}^2$ and equality holds for all $u \in \Omega_m$ if $K = 0$.

If instead $u \in \Omega_1$ we have $u = f_1 + f_{-1}$ with $f_j \in \Lambda_j$, and Lemma \ref{lemma_gk_energy_identity} again gives 
\begin{align*}
\norm{X_- u}^2 &= \norm{\eta_- f_1 + \eta_+ f_{-1}}^2 \leq 2(\norm{\eta_- f_1}^2 + \norm{\eta_+ f_{-1}}^2) \\
 &= 2 \left[ \norm{\eta_+ f_1}^2 + \norm{\eta_- f_{-1}}^2 + \frac{1}{2} ((K f_1, f_1) + (K f_{-1}, f_{-1})) \right].
\end{align*}
Since $K \leq 0$ we get $\norm{X_- u}^2 \leq 2 \norm{X_+ u}^2$. If $K=0$ we have equality if and only if $\eta_- f_1 = \eta_+ f_{-1}$. Identifying $\Omega_1$ with $1$-forms on $M$, this means that the $1$-form $f_1 - f_{-1}$ is divergence free and thus $f_1 - f_{-1} = *d a_0 + h$ for some $a_0 \in C^{\infty}(M)$ and some harmonic $1$-form $h$. Consequently, if $K=0$ then equality holds exactly when $u = d a_0 + h$ for some $a_0 \in C^{\infty}(M)$ and some harmonic $1$-form $h$. (Note that $h = h_1 + h_{-1}$ is a harmonic $1$-form if and only if $\eta_- h_1 = \eta_+ h_{-1} = 0$.)

The inequalities for $B$ follow as in Section \ref{section_beurling}.
\end{proof}

\begin{Remark}{\rm We note that the inequality $\norm{X_{-}u}\leq \sqrt{2}\norm{X_{+}u}$ for $u\in\Omega_1$ differs by the factor $\sqrt{2}$ with inequality (5.2) in \cite[p. 173]{GK2} for $n=2$ and $p=1$ which gives $\norm{X_{-}u}\leq \norm{X_{+}u}$. Since our inequality in Lemma \ref{lemma_twodim_beurling2} is shown to be sharp in the flat case this indicates an algebraic mistake in the calculation of the constants in \cite{GK2}.}
\end{Remark}

\vspace{10pt}

\noindent {\bf Pestov and Guillemin-Kazhdan energy identities.} We conclude this section by discussing the relation between two basic energy identities. The Pestov identity from Proposition \ref{prop_pestov} takes the following form in two dimensions:
$$
\norm{VXu}^2 = \norm{XVu}^2 - (KVu,Vu) + \norm{Xu}^2, \qquad u \in C^{\infty}(SM).
$$
The Guillemin-Kazhdan energy identity in Lemma \ref{lemma_gk_energy_identity} looks as follows:
$$
\norm{\eta_- u}^2 = \norm{\eta_+ u}^2 - \frac{i}{2} (K V u, u), \qquad u \in C^{\infty}(SM).
$$
As discussed in \cite{PSU1}, the Pestov identity is essentially the commutator formula $[XV,VX] = -X^2 + VKV$, whereas the Guillemin-Kazhdan identity follows from the commutator formula $[\eta_+, \eta_-] = \frac{i}{2} K V$.

We now show that the Pestov identity applied to $u \in \Lambda_k$ is just the Guillemin-Kazhdan identity for $u \in \Lambda_k$. Indeed, we compute 
$$
\norm{VXu}^2 = \norm{V\eta_+ u}^2 + \norm{V \eta_- u}^2 = (k+1)^2 \norm{\eta_+ u}^2 + (k-1)^2 \norm{\eta_- u}^2
$$
and 
$$
\norm{XVu}^2 - (KVu,Vu) + \norm{Xu}^2 = k^2(\norm{\eta_+ u}^2 + \norm{\eta_- u}^2) + ik(KVu,u) + \norm{\eta_+ u}^2 + \norm{\eta_- u}^2.
$$
The Pestov identity and simple algebra show that 
$$
2k(\norm{\eta_+ u}^2 - \norm{\eta_- u}^2) = ik(KVu,u)
$$
This is the Guillemin-Kazhdan identity if $k \neq 0$.

In the converse direction, assume that we know the Guillemin-Kazhdan identity for each $\Lambda_k$, 
$$
\norm{\eta_+ u_k}^2 - \norm{\eta_- u_k}^2 = \frac{i}{2}(KVu_k,u_k), \qquad u \in \Lambda_k.
$$
Multiplying by $2k$ and summing gives 
$$
\sum 2k(\norm{\eta_+ u_k}^2 - \norm{\eta_- u_k}^2) = \sum ik (KVu_k,u_k).
$$
On the other hand, the Pestov identity for $u = \sum_{k=-\infty}^{\infty} u_k$ reads 
\begin{multline*}
\sum k^2 \norm{\eta_+ u_{k-1} + \eta_- u_{k+1}}^2 \\
 = \sum (\norm{\eta_+(Vu_{k-1}) + \eta_-(Vu_{k+1})}^2 + ik (KVu_k, u_k) + \norm{\eta_+ u_{k-1} + \eta_- u_{k+1}}^2).
\end{multline*}
Notice that 
$$
k^2 \norm{\eta_+ u_{k-1} + \eta_- u_{k+1}}^2 = k^2 (\norm{\eta_+ u_{k-1}}^2 + \norm{\eta_- u_{k+1}}^2) + 2k^2 \mathrm{Re} (\eta_+ u_{k-1}, \eta_- u_{k+1})
$$
and 
\begin{multline*}
\norm{\eta_+(Vu_{k-1}) + \eta_-(Vu_{k+1})}^2 + \norm{\eta_+ u_{k-1} + \eta_- u_{k+1}}^2 \\
 = (k^2-2k+2) \norm{\eta_+ u_{k-1}}^2 + (k^2+2k+2) \norm{\eta_- u_{k+1}}^2 + 2 k^2 \mathrm{Re} (\eta_+ u_{k-1}, \eta_- u_{k+1}).
\end{multline*}
Thus the Pestov identity is equivalent with 
$$
\sum \left[ (2k-2) \norm{\eta_+ u_{k-1}}^2 - (2k+2) \norm{\eta_- u_{k+1}}^2 \right] = \sum ik (KVu_k, u_k).
$$
This becomes the summed Guillemin-Kazhdan identity after relabeling indices.

\bibliographystyle{alpha}

\end{document}